\documentclass[10pt]{amsart}
\usepackage[all]{xy} 
\usepackage{amsfonts, amssymb}
\usepackage{tikz-cd} 
\usepackage{graphics, graphicx}
\usepackage{color}\definecolor{Green}{rgb}{0.0, 0.5, 0.0}

\newtheorem{thm}{Theorem}[section]
\newtheorem{cor}[thm]{Corollary}
\newtheorem{prop}[thm]{Proposition}
\newtheorem{lem}[thm]{Lemma}
\theoremstyle{definition}
\newtheorem{defn}[thm]{Definition}
\newtheorem{ex}[thm]{Example}

\newtheorem{rem}[thm]{Remark}
\newtheorem*{ack}{Acknowledgments}

\numberwithin{equation}{section}

\newcommand{\R}{\mathbb R}

\newcommand{\oo}{\bigcirc}
\newcommand{\ii}{\,|\,}
\newcommand{\e}{\mathfrak e}
\newcommand{\ve}{\mathfrak v}

\newcommand{\ri}{{\bf ri}}
\newcommand{\rbd}{{\bf rbd}}

\title{The Combinatorics of Directed Planar Trees}

\author[K.~Poirier]{Kate~Poirier}
  \address{Kate Poirier,
  Department of Mathematics, New York City College of Technology, City University of New York, 300 Jay Street, Brooklyn, NY 11201}
  \email{kpoirier@citytech.cuny.edu}

\author[T.~Tradler]{Thomas~Tradler}
  \address{Thomas Tradler,
  Department of Mathematics, New York City College of Technology, City University of New York, 300 Jay Street, Brooklyn, NY 11201}
  \email{ttradler@citytech.cuny.edu}

\begin{document}
\maketitle
\begin{abstract}
We give a geometric realization of the polyhedra governed by the structure of associative algebras with co-inner products, or more precisely, governed by directed planar trees. Our explicit realization of these polyhedra, which include the associahedra in a special case, shows in particular that these polyhedra are homeomorphic to balls. We also calculate the number of vertices of the lowest generalized associahedra, giving appropriate generalizations of the Catalan numbers. 
\end{abstract}

\setcounter{tocdepth}{1}
\tableofcontents

\section{Introduction}\label{SEC:Introduction}

The associahedron, or Stasheff polytope, is a convex polytope whose cellular structure is determined by the combinatorics of planar, rooted trees. In \cite{S1, S2} Jim Stasheff used these polytopes to study H-spaces up to homotopy, and in particular gave a geometric realization of the associahedron inside a cube. The associahedra appear in many settings in mathematics due to their fundamental definition, and here we note one instance which is relevant for our purposes, namely, the fact that their cellular chains may be used to define the operad of $A_\infty$-algebras (see \cite{MSS}), giving a resolution of the associative operad.

In this note, we describe a variation of these polytopes, which originally grew out of an attempt to model algebraically string topology operations as defined by Moira Chas and Dennis Sullivan in \cite{CS}. In fact, to do so, an essential ingredient consists of a model for the Poincar\'e duality structure of the underlying space. For example, in \cite{T}, the Poincar\'e duality structure was modeled via a non-degenerate, invariant inner-product with higher homotopies (which were called homotopy inner products). More generally, if one considers an invariant \emph{co}-inner product (with higher homotopies), one may drop the non-degeneracy condition, and still obtain string topology-like operations; this was defined in an algebraic setting in \cite{TZ}. In this setup one requires $n$-to-$m$-operations (i.e. maps $A^{\otimes n}\to A^{\otimes m}$) for each corolla having a \emph{cyclic} order on its inputs and outputs (satisfying the usual edge expansion conditions). Such an algebraic structure on a space $A$ was called a $V_\infty$ algebra in \cite{TZ}. 
It is our aim with this paper and two follow-up papers to clarify the combinatorics of this structure as well as identify operadic underpinnings of $V_\infty$ algebras, and furthermore identify the induced space of string topology operations with other models of this space of operations.

In this paper, we take a first step toward analyzing the structure of $V_\infty$ algebras. Using the combinatorics of directed planar trees with a cyclic order $\alpha$ on their exterior vertices (Definition \ref{DEF:directed-trees}), we define a cell complex $Z_\alpha$, our generalization of the associahedron, whose cells are indexed by precisely those trees.  This is done using and adding onto the well-known secondary polytope construction of the associahedron defined by Gelfand, Kapranov, and Zelevinsky (see e.g. \cite{GKZ}). We show in Section \ref{SEC:ConstructingZalpha}, that $Z_\alpha$ is homeomorphic to a disk, or more precisely, we show the following.
\setcounter{section}{3} \setcounter{thm}{9}
\begin{thm}
The space $Z_\alpha$ has the structure of a cell complex where the cells are given by the subspaces $Z_T$ for $T$ in $\mathcal{T}_\alpha$.
This structure is a cellular subdivision of the product of an associahedron and a simplex $K_{n_\alpha-1}\times \Delta^{k_\alpha-1}$ in $\R^{n_\alpha}\times \R^{k_\alpha}$, each with their own natural cell complex structures.
\end{thm}
In the case where there are exactly two outgoing edges, and $\ell$ and $m$ incoming edges (between the two outgoing edges) these polyhedra are precisely the pairahedra as defined in \cite{T}; see Example \ref{EX:polyhedra}\eqref{EX:polyhedra-OO} below.

In Section \ref{SEC:Vertices} we investigate some of the combinatorics of $Z_\alpha$ by studying the number $C(\alpha)$ of vertices of $Z_\alpha$. We give a recursive formula for calculating $C(\alpha)$ in Proposition \ref{PROP:C(alpha)-recursion}. In the case where there is exactly one outgoing edge and, say, $\ell$ incoming edges these numbers are, of course, well known to be the Catalan numbers $C_{\ell-1}=\frac 1 \ell {2(\ell-1) \choose \ell-1}$. In the case where there are exactly two outgoing edges, and $\ell$ and $m$ incoming edges between them, we denote these numbers by $c_{\ell,m}$ and calculate them explicitly in Proposition \ref{PROP:c-lm}.
\setcounter{section}{4} \setcounter{thm}{2}\begin{prop}
\[
c_{\ell,m}= {2(\ell+1)\choose \ell+1} {2(m+1)\choose m+1}\cdot \frac{(\ell+1)(m+1)}{2(\ell+m+1)(\ell+m+2)}
\]
\end{prop}

In the same way the associahedra $K_n$ are related to the concept of associativity, our new cell complexes $Z_\alpha$ are related to the concept of associtivity together with a symmetric and invariant co-inner product. 
For this reason, we refer to our spaces $Z_\alpha$ as {\em assocoipahedra}.
In fact, in a follow-up paper, we are planning to show that the cell complexes $Z_\alpha$ can be used to define the dioperad $V_\infty$, and with this show that $V_\infty$ is a resolution of a dioperad $V$ governing associative algebras with symmetric and invariant co-inner products.
This can then be used to show that the dioperad $V$ is Koszul. Furthermore, we are planning to show that the induced space of string topology operations for a $V_\infty$ algebra, as defined in \cite{TZ}, is homotopy equivalent to the more topological space of string topology operations defined in \cite{DPR}.  All of these follow-up results will however crucially rely on the fact that the cell complexes $Z_\alpha$ are homeomorphic to disks, which is the content of this paper.

\begin{ack}
The second author was supported in part by a grant from The City University of New York PSC-CUNY Research Award Program.
We thank Anton Dochtermann for useful comments about this paper.
\end{ack}
\setcounter{section}{1}

\section{Preliminaries on Directed Planar Trees}\label{SEC:PlanarTrees}

In this section, we define the precise notion of  ``$\alpha$-trees,'' directed trees that we consider in this paper. We show that the data of an $\alpha$-tree $T$ can be equivalently written as a Stasheff-type tree $S_T$ (with one outgoing edge only) and an ``essential spine'' $E_T$ that keeps the information of the directions of edges of $T$. Using this decomposition, we will show in the next section how 
the set of directed planar trees can be geometrically realized as a cellular subdivision of a Cartesian product of an associahedron and a simplex, $K_{n-1}\times \Delta^{k-1}$.

Let $\alpha$ be a sequence of incoming ``$\ii$'' and outgoing ``$\oo$'' labels; for example $\alpha=(\oo\oo\ii\ii\ii\oo\ii\oo)$. Let $k_\alpha$ be the number of outgoing labels of $\alpha$, $\ell_\alpha$ be the number of incoming labels, and $n_\alpha=k_\alpha+\ell_\alpha$ be the total number of labels. For $j=1,\dots, n_\alpha$, we denote by $\alpha(j)\in\{\ii, \oo\}$ the $j$th element in the sequence $\alpha$.

We can obtain an $\alpha$ as above from a directed planar tree with a chosen first external vertex as follows. Each external vertex can be given an $\ii$ or $\oo$ label depending on whether the adjacent edge points 
away from that vertex (coming into the tree) or towards that vertex (going out from the tree); see Figure \ref{FIG:alpha-tree}.
A linear order of these external vertices is given by the clockwise order determined by the plane, together with the choice of first external vertex.
We abuse notation by referring to both these external edges and vertices as incoming or outgoing.

\begin{defn}\label{DEF:directed-trees}
An $\alpha$-tree is a directed planar tree with at least one interior vertex, such that:
\begin{enumerate}
\item\label{choice-vertex}
there is a choice of one of the exterior vertices,
\item\label{exterior-cyclic-alpha-order}
the sequence of labels of the exterior vertices as incoming or outgoing according to the above procedure, starting from the chosen exterior vertex in \eqref{choice-vertex}, coincides with the given $\alpha$ 
\item\label{int-vertex}
every interior vertex has at least one outgoing edge, and
\item\label{no-one-one-vertex}
there are no bivalent vertices with one incoming and one outgoing edge.
\end{enumerate}
Note that by \eqref{int-vertex} and \eqref{no-one-one-vertex}, the only permitted bivalent vertices are those with two outgoing edges.
\begin{figure}[h]
\[ 
\includegraphics[scale=.8]{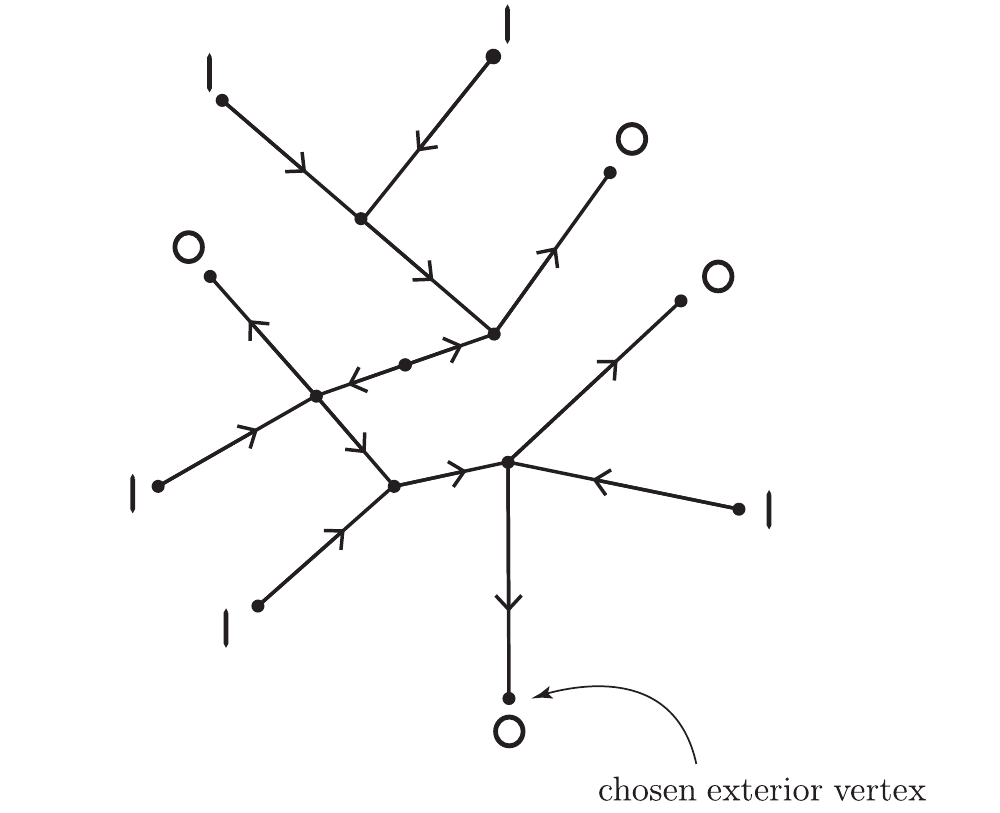}
\]
\caption{An $\alpha$-tree $T$, with $\alpha=(\oo\ii\ii\oo\ii\ii\oo\oo\ii)$}\label{FIG:alpha-tree}
\end{figure}
We define a Stasheff-type tree to be a $(\oo\ii\ii\dots\ii\ii)$-tree with $1$ outgoing exterior vertex and  $\ell\geq 2$ incoming exterior vertices. We denote the set of $\alpha$-trees by $\mathcal T_{\alpha}$.
\end{defn}

\begin{rem}
Conditions~(\ref{choice-vertex}) through~(\ref{no-one-one-vertex}) imply that the interior edges of a Stasheff-type tree must be directed toward the outgoing external vertex; every interior vertex has exactly one outgoing edge.
\end{rem}

Let $T$ and $T'$ be $\alpha$-trees. Then, $T'$ is called an edge expansion of $T$, if there are interior edges $e_1,\dots,e_k$ in $T'$, so that contracting these edges in $T'$ yields $T$, i.e. $T=T'/(e_1,\dots,e_k)$. (Note, that a collapse of any interior edge of an $\alpha$-tree yields again an $\alpha$-tree.) We define the corolla $T_{\alpha}$ to be the unique $\alpha$-tree with no internal edge. Then, every $\alpha$-tree $T$ is an edge expansion of the corolla $T_{\alpha}$.

In the next section, we will define a finite cell complex $Z_{\alpha}$ which is homeomorphic to a closed disk, such that the cells of $Z_{\alpha}$ are indexed by $\mathcal T_{\alpha}$. In fact, we will give an explicit geometric realization of $Z_\alpha$ as a cellular subdivision of a product $K_{n_\alpha-1}\times \Delta^{k_\alpha-1}$ of an associahedron $K_{n_\alpha-1}$ and a simplex $\Delta^{k_\alpha-1}$. Since the associahedron $K_{n_\alpha-1}$ is of dimension $n_\alpha-3$, we see that $Z_\alpha$ is a cell complex of dimension $(n_\alpha-3)+(k_\alpha-1)=n_\alpha+k_\alpha-4=\ell_\alpha+2k_\alpha-4$.

We now present a way to rewrite an $\alpha$-tree $T$ in an equivalent combinatorial way, given by a Stasheff-type tree $S_T$, and another tree, called its essential spine $E_T$.
\begin{defn} Consider an $\alpha$-tree $T$. The {\em underlying Stasheff tree $S_T$ of $T$} is the tree of Stasheff-type, obtained by removing all orientations on the edges, removing all bivalent vertices (which necessarily have two edges that are outgoing from this vertex)
and replacing these two edges with a single edge. 
The choice of exterior vertex of $S_T$ (the ``\emph{root}'' of $S_T$) will be the same as for $T$, and all other external vertices are incoming vertices with a ``flow'' to the unique exterior vertex; see e.g. Figure \ref{FIG:treeT-S(T)}.
\begin{figure}[h]
\[
 \includegraphics[scale=0.8]{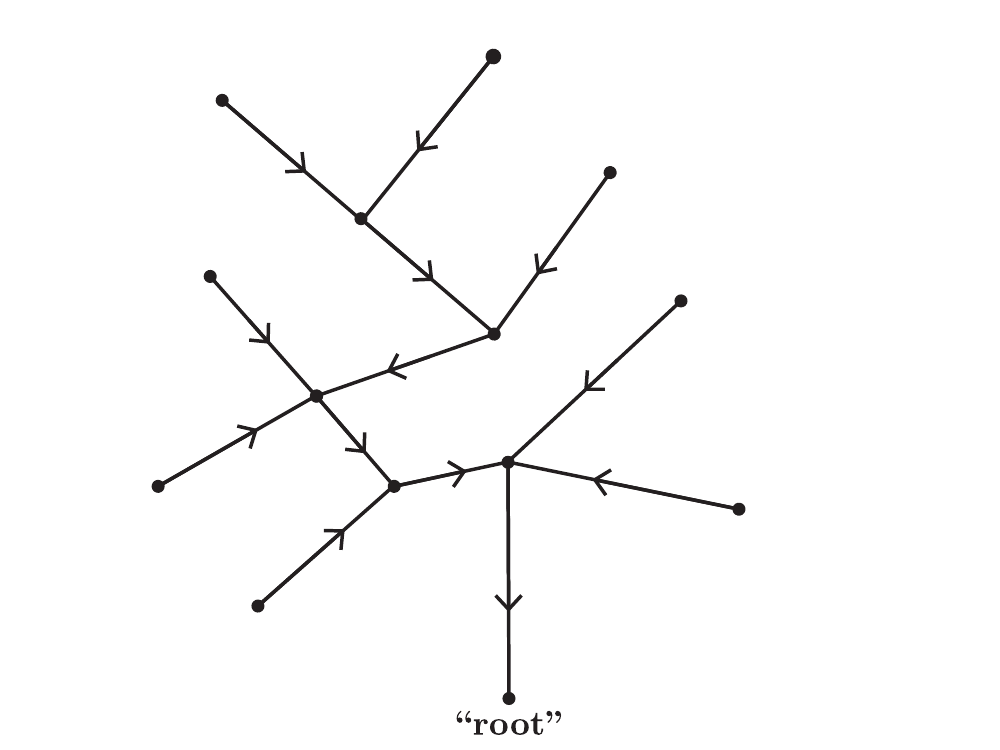}
 \]
\caption{The underlying Stasheff tree $S_T$ for the $\alpha$-tree $T$ from Figure \ref{FIG:alpha-tree}}\label{FIG:treeT-S(T)}
\end{figure}

In other words, the underlying Stasheff tree completely disregards the orientation on the edges of the original tree $T$, and only has the information of the underlying undirected tree of $T$ and the choice of external vertex of $T$. To recover the information of these orientations, we define the {\em spine} $P_T$ of $T$ as follows. Any two external outgoing edges of $T$ can be connected by a unique shortest path in $T$. The union of vertices and edges of these paths over all pairs of outgoing edges (including their orientations in $T$) will be denoted by $P_T$; see e.g. Figure \ref{FIG:treeT-P(T)}. 
There is a canonical choice of one of the external vertices of the spine $P_T$, namely the vertex with the first outgoing label in the sequence $\alpha$.
\begin{figure}[h]
\[
 \includegraphics[scale=0.8]{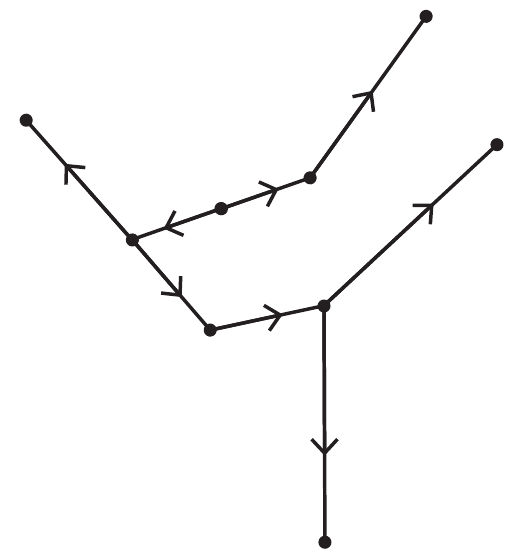}
\]
\caption{
The spine $P_T$ for the $\alpha$-tree $T$ from Figure \ref{FIG:alpha-tree} 
}\label{FIG:treeT-P(T)}
\end{figure}
Clearly, placing the spine $P_T$ on the Stasheff tree $S_T$ and directing all edges that are not on the spine toward $P_T$ will recover the tree $T$ from $S_T$ and $P_T$.

Since $S_T$ has all bivalent vertices removed, it will be useful to do the same for $P_T$. In fact, we define the {\em essential spine $E_T$ of $T$} to be the tree obtained from $P_T$ by removing those vertices which were bivalent in $T$ (and which necessarily had two outgoing edges in $T$), and 
replacing these two edges with a single edge. We then give
 each edge of $E_T$ one of three kinds of labels. If an edge of $E_T$ is an original edge of $T$, then we label it with the orientation provided by $T$. If an edge of $E_T$ is obtained by combining two edges at a bivalent vertex of $T$, then the new edge will be labeled by a new symbol ``$\leftrightarrow$''. Thus, each edge of $E_T$ is labeled either with one of the two orientations of the edge, or with the symbol ``$\leftrightarrow$''. An example is shown in Figure \ref{FIG:treeT-E(T)}.
\begin{figure}[h]
\[ 
\includegraphics[scale=0.8]{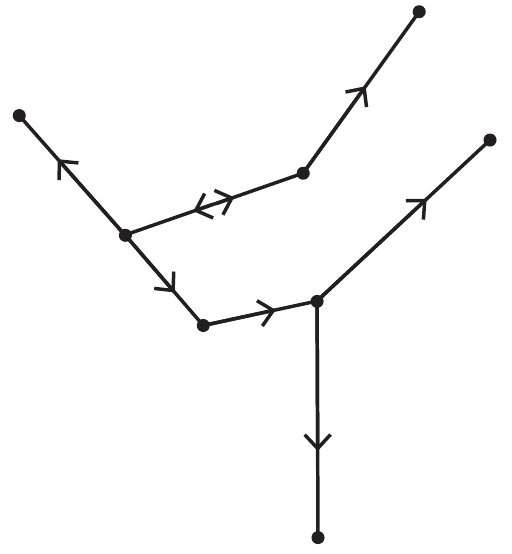}
\]
\caption{
The essential spine $E_T$ for the $\alpha$-tree $T$ from Figure \ref{FIG:alpha-tree}}\label{FIG:treeT-E(T)}
\end{figure}

More generally, we define an essential spine as follows.
\begin{defn}\label{DEF:essential-spine}
An {\em essential spine} is a planar tree $E$, such that each edge is labeled with one of three possible labels: one of the two possible orientations on its edges or  the symbol ``$\leftrightarrow$.'' We require the following conditions for an essential spine:
\begin{enumerate}
\item\label{ITEM:choice-alpha} there is a choice of one of the exterior vertices,
\item\label{ITEM:label-edges} each external edge of $E$ is labeled either with its outgoing direction or with the symbol ``$\leftrightarrow$,'' and
\item\label{ITEM:internal-one-outgoing} each internal vertex of $E$ must have at least one outgoing edge.
\end{enumerate}
\end{defn}
Note, that the essential spine $E_T$ of an $\alpha$-tree $T$ is indeed an essential spine. Note also, that there is no analogue of item \eqref{no-one-one-vertex} in Definition \ref{DEF:directed-trees}, since the essential spine $E_T$ may have bivalent vertices with one incoming and one outgoing edge, see e.g. Figure \ref{FIG:treeT-E(T)}.

We call a Stasheff-type tree $S$ and an essential spine $E$ {\em compatible}, or more precisely, {\em $(j_1,\dots,j_k)$-compatible}, if the underlying tree of $E$ (ignoring the labeling of the edges) is precisely the subtree of $S$ obtained by connecting the external edges of $S$ at positions $j_1, \dots, j_k$ via their shortest paths. This means, in particular, that $E$ must contain precisely all edges and vertices from $S$ obtained by connecting these external edges at positions $j_1, \dots, j_k$. Again, note that for an $\alpha$-tree $T$ whose outgoing labels are at positions $(j_1,\dots,j_{k_\alpha})$, the underlying Stasheff tree $S_T$ and the essential spine $E_T$ are indeed $(j_1,\dots,j_{k_\alpha})$-compatible. For example, the Stasheff-type tree $S_T$ from Figure \ref{FIG:treeT-S(T)} and the essential spine $E_T$ from Figure \ref{FIG:treeT-E(T)} are $(1,4,7,8)$-compatible.

For $\alpha$ with outgoing labels at positions $(j_1,\dots,j_{k_\alpha})$, we denote by 
\begin{eqnarray*}
\mathcal {SE}_{\alpha}&=& \{(S,E): \text{$S$ is a Stasheff tree with $n_\alpha$ external vertices, and $E$ is}\\
&&\hspace{.6in} \text{an essential spine which is $(j_1,\dots,j_{k_\alpha})$-compatible with $S$}\}.
 \end{eqnarray*}
Then, the above construction of $S_T$ and $E_T$ provides a map $f:\mathcal T_{\alpha}\to \mathcal {SE}_{\alpha},  f(T)=(S_T,E_T)$.
\end{defn}
We now have the following lemma.
\begin{lem}\label{LEM:T-SE-bijection}
The map $f:\mathcal T_{\alpha}\to \mathcal {SE}_{\alpha}, f(T)=(S_T,E_T)$ is a bijection, where the inverse $f^{-1}(S,E)$ is given by combining $S$ and $E$ by placing $E$ on $S$ with its induced labels (either use orientation from $E$, or introduce a new bivalent vertex for ``$\leftrightarrow$''; edges in $S$ that are not in $E$ are all labeled as incoming edges).
\end{lem}
\begin{proof}
First, note that for $(S,E)\in \mathcal{SE}_\alpha$, $(j_1,\dots,j_{k_\alpha})$-compatibility implies that $E$ can be placed on $S$ when ignoring the labels. Now, $f^{-1}(S,E)$ is obtained from $S$ by placing the labels from $E$ on the edges of $S$, orienting all edges not in $E$ toward $E$, and changing any edge with label $\leftrightarrow$ to two edges with outgoing orientations from the new (bivalent) vertex. Note, that $f^{-1}(S,E)$ has bivalent vertices only when there was an edge in $E$ labeled with $\leftrightarrow$, since $S$ had no bivalent vertices to begin with. To check that $f^{-1}(S,E)$ is an $\alpha$-tree, it remains to check  condition \eqref{int-vertex} from Definition \ref{DEF:directed-trees}, namely that every interior vertex has at least one outgoing edge. This follows from the condition of $E$ that each vertex must have at least one outgoing edge. This shows that $f^{-1}(S,E)$ is a well-defined $\alpha$-tree of type $(n,k)$. It is now immediate to check that $f\circ f^{-1}=id$ and $f^{-1}\circ f=id$.
\end{proof}

We now transfer the notion of ``edge expansion'' to the space $\mathcal{SE}_\alpha$ of Stasheff trees $S$ with compatible essential spines $E$.

\begin{defn}\label{DEF:formal-expansion-and-dim}
Let $(S,E)$ and $(S',E')$ be elements of $\mathcal {SE}_{\alpha}$. We call $(S',E')$ a {\em formal edge expansion of $(S,E)$ by one edge} if either:
\begin{enumerate}
\item\label{expansion-S'=S}
 $S'=S$, and $E'$ is obtained from $E$ by changing one label of an edge from a direction to the symbol $\leftrightarrow$,
\item\label{expansion-E'=E}
 $S'$ is a one edge expansion of $S$, and the new edge of $S'$ does not appear in the essential spine, and $E'=E$, or
\item\label{expansion-new-leftright}
 $S'$ is a one edge expansion of $S$, and the new edge of $S'$ does appear in the essential spine, and the new edge in $E'$ is labeled by a direction (but not $\leftrightarrow$).
\end{enumerate}

More generally, $(S',E')$ a {\em formal edge expansion of $(S,E)$} is there is a sequence of formal edges expansions by one edge $(S',E')\rightsquigarrow(S^{(1)},E^{(1)})\rightsquigarrow\dots\rightsquigarrow(S^{(p)},E^{(p)})\rightsquigarrow(S,E)$.
\end{defn}

\begin{lem} 
Let $T,T'\in \mathcal T_\alpha$. Then $T'$ is an edge expansion of $T$ if and only if $f(T')$ is a formal edge expansion of $f(T)$. 
\end{lem}
\begin{proof}
Assume that $T=T'/e$ is obtained from $T'$ by collapsing only one edge $e$. Let $f(T)=(S,E)$ and $f(T')=(S',E')$. If the collapsed edge $e$ does not appear in the spine of $T'$, then $E'=E$ and $S'$ must be an edge expansion of $S$, i.e. \eqref{expansion-E'=E} of Definition \ref{DEF:formal-expansion-and-dim}. If the new edge does appear in $E'$, then either it created a bivalent vertex with two outgoing edges (i.e. \eqref{expansion-S'=S} of Definition \ref{DEF:formal-expansion-and-dim}) or it did not (i.e. \eqref{expansion-new-leftright} of Definition \ref{DEF:formal-expansion-and-dim}), but in either case it is a formal edge expansion. Iterating this for multiple edges gives the result.
\end{proof}

We call an $\alpha$-tree $T\in\mathcal T_\alpha$ maximally expanded, if there are no edge expansions of $T$. A similar definition applies to $(S,E)\in \mathcal{SE}_\alpha$.
\begin{lem}
$(S,E)\in \mathcal{SE}_\alpha$ is maximally expanded iff all internal vertices of $S$ are trivalent and each interior vertex of $E$ has exactly one outgoing edge.
\end{lem}
\begin{proof}
The condition on $S$ is necessary since any non-trivalent internal vertex can be further expanded. So, assume now that $S$ has only trivalent internal vertices. Now, $E$ is an essential spine (Definition \ref{DEF:essential-spine}), so that all edges of $E$ are labeled with a direction or $\leftrightarrow$. The only possible edge expansion of $(S,E)$ occurs by changing a direction of $E$ to a symbol $\leftrightarrow$ (Definition \ref{DEF:formal-expansion-and-dim} \eqref{expansion-S'=S}). Now, if any of the interior vertices of $E$ has more than one outgoing edge, then one of these edges can be changed to a label $\leftrightarrow$, giving an edge expansion $(S,E')$ of $(S,E)$. Thus, $(S,E)$ is maximally expanded exactly when each interior vertex of $E$ has precisely one outgoing edge.
\end{proof}

\section{Geometric Realization of the Set of Directed Planar Trees}\label{SEC:ConstructingZalpha}

In this section, we define the assocoipahedron $Z_\alpha$, a cell complex
 whose cells are labeled by 
the set of $\alpha$-trees in $\mathcal {T}_\alpha$. We give an explicit geometric realization of this cell complex as a subdivision of a product $K_{n-1}\times \Delta^{k-1}$ of an associahedron and a simplex, which uses and extends the secondary polytope construction; see \cite{GKZ}. Our main Theorem \ref{THM:Zalpha-cell-complex} states that this construction gives a well-defined cell complex.

To construct the polytope $Z_\alpha$ we first recall how one can construct the associahedron from a Stasheff-type tree $S$. There are various (non-equivalent) ways to construct the associahedron; we refer the reader to \cite{CSZ} for an interesting comparison among these constructions. In this paper, we will mainly use the secondary polytope construction (Definition \ref{DEF:secondary-polytope}) defined by Gelfand, Kapranov and Zelevinsky in \cite{GKZ} to parametrize the associahedron; however other constructions would work as well for our construction; see Remark \ref{REM:Loday-construction-Kalpha} below.

Recall that we have fixed $\alpha$, which is a sequence of $n_\alpha$ many incoming and outgoing labels, for which we will assume that $n_\alpha\geq 3$. For ease of notation, we will simply write $n=n_\alpha$ in the next definition. Now, additionally fix a convex $n$-gon $Q\subseteq \R^2$, given as the convex hull $Q:= conv(q_1,\dots, q_{n})$ of vertices $q_1,\dots q_{n}\in \R^2$ such that no three of these vertices are collinear. We assume $q_1,\dots, q_{n}$ appear in this cyclic order in the boundary of $Q$, and we choose the line segment $\overline{q_1q_{n}}$ as the base side of $Q$.
\begin{defn}\label{DEF:secondary-polytope}
Each Stasheff-type tree $S$ can be uniquely represented as a subdivision of $Q$ by non-intersecting diagonals; see Figure \ref{FIG:S-from-Q}. 
\begin{figure}[h]
\[
 \includegraphics[scale=0.8]{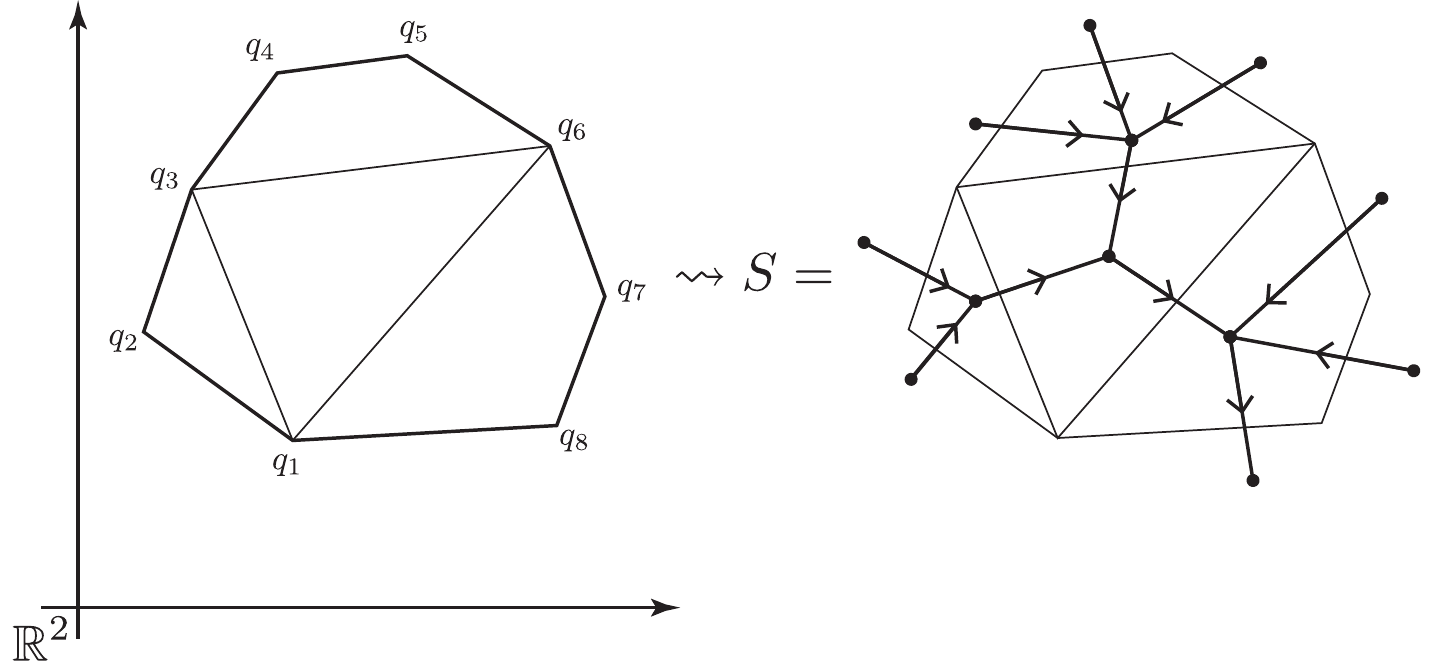}
\]
\caption{The Stasheff-type tree $S$ as the transversal tree from a subdivision of $Q$ by non-intersecting diagonals}\label{FIG:S-from-Q}
\end{figure}
Note, that the maximally expanded Stasheff-type trees (whose internal vertices are all trivalent) correspond exactly to triangulations of $Q$; i.e. a subdivision into $n-2$ many triangles whose vertices are all coming from $q_1,\dots q_{n}$. 

Now, let $S$ be a maximally expanded Stasheff-type tree with corresponding triangulation $t=t(S)$ of $Q$. For each vertex $q_j$ (where $j=1,\dots, n$), let $Star_t(j)$ be the union of triangles in $t$ that have $q_j$ as a vertex, and denote by $area(Star_t(j))$ its area. Then, we define the vector $v_t\in\R^{n}$ by setting
\[
v_t:=\sum_{j=1}^{n} area(Star_t(j))\cdot e_j
\]
where $\{e_j\}_{j=1,\dots,n}$ is the standard basis of $\R^{n}$. With this, the secondary polytope $K_Q\subseteq\R^{n}$ is defined to be the convex hull of the vectors $v_t$ ranging over all triangulations of $Q$:
\[
K_{Q}:= conv(\{v_t: t\text{  is a triangulation of }Q\}).
\]
It is well-known, that $K_Q\subseteq \R^{n}$ is an $n-3$ dimensional convex polytope, which is a geometric representation of the associahedron $K_{n-1}$; see e.g. \cite[Section 7.3.B, p. 237ff]{GKZ}.
\end{defn}

Using the above construction of the secondary polytope, if $T\in \mathcal T_\alpha$ is a maximally expanded $\alpha$-tree, we next define a vector $w_T\in \R^{k_\alpha}$. The convex hull of all these vectors $w_T$ will be denoted by $\Delta_Q=conv(\{w_T\}_T)\cong \Delta^{k_\alpha-1}\subseteq \R^{k_\alpha}$, and our polytope $Z_\alpha$ will be given as $Z_\alpha:=K_Q\times \Delta_Q\cong K_{n_\alpha-1}\times \Delta^{k_\alpha-1}\subseteq  \R^{n_\alpha}\times \R^{k_\alpha}$.
\begin{defn}\label{DEF:w_T}
Let $T\in \mathcal T_\alpha$ be maximally expanded and write $f(T)=(S_T,E_T)$ for the corresponding Stasheff-type tree $S_T$ and essential spine $E_T$; see Section \ref{SEC:PlanarTrees}. Since $S_T$ is maximally expanded, there is an associated triangulation $t$ of the $n_\alpha$-gon $Q$ associated with $S_T$ as described in Definition \ref{DEF:secondary-polytope} and thus there is a vector $v_T:=v_t\in \R^{n_\alpha}$.

Now, using the essential spine $E_T$ we define a vector $w_T\in \R^{k_\alpha}$. Since $T$ is maximally expanded, each internal vertex of $E_T$ has exactly one outgoing edge. If we cut the tree $E_T$ at all edges labeled by $\leftrightarrow$, then we obtain subtrees, each of which has a flow from incoming edges to one of the external outgoing edges of $E_T$ (since each of the internal vertices of $E_T$ must have an outgoing edge); see Figure \ref{FIG:Cut-ET}.
\begin{figure}[h]
\[
 \includegraphics[scale=0.8]{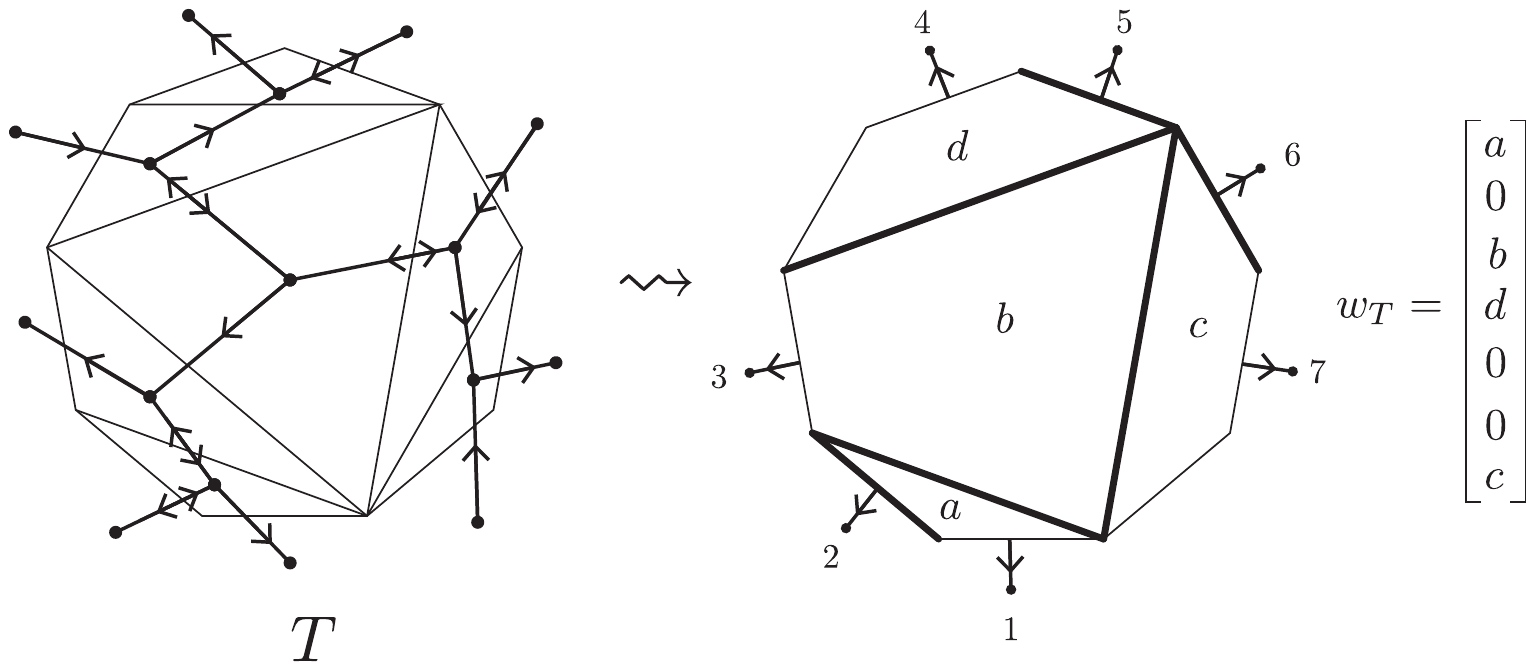}
\]
\caption{Decomposition of $T$ by cutting edges with $\leftrightarrow$ labels. Here, $a, b, c, d$ are the areas of the displayed regions, and $w_T$ is the corresponding vector in $\R^7$}\label{FIG:Cut-ET}
\end{figure}

Thus, the essential spine gives a decomposition of $T$ into subtrees such that each subtree has exactly one of the $k_\alpha$ outgoing vertices of $T$. Note, that each edge in $E_T$ also corresponds via the triangulation given by $S_T$ to either a diagonal in $Q$ or to one of the boundary line segments of $Q$; see Figure \ref{FIG:Cut-ET}. Thus, cutting the $n_\alpha$-gon $Q$ along those diagonals yields $k_\alpha$ many subpolygons, each of which corresponds to exactly one of the outgoing vertices of $T$. Here we need include zero-area segments as degenerate subpolygons; see e.g. the outgoing vertices $2$, $5$, or $6$ in Figure \ref{FIG:Cut-ET}. For $i=1,\dots, k_\alpha$, denote by $Q_T(i)$ the ``subpolygon'' associated to the $i$th outgoing vertex of $T$, and let $area(Q_T(i))\geq 0$ be its area. Then, define the vector $w_T\in \R^{k_\alpha}$ by setting
\begin{equation}\label{EQN:wT}
w_T:=\sum_{i=1}^{k_\alpha} area(Q_T(i))\cdot e_i
\end{equation}
where $\{e_i\}_{i=1,\dots,k_\alpha}$ is the standard basis of $\R^{k_\alpha}$. Then, we define $\Delta_{Q}\subseteq \R^{k_\alpha}$ as the convex hull of the vectors $w_T$ ranging over all maximally expanded $\alpha$-trees $T$:
\[
\Delta_Q:=conv(\{w_T: T\text{ is a maximally expanded $\alpha$-tree}\}).
\]

Finally, for this choice of $Q$, we define our space to be
\[
Z_\alpha:=Z_{Q,\alpha}:=K_Q\times \Delta_Q\quad\subseteq \R^{n_\alpha}\times\R^{k_\alpha}.
\]
In Theorem~\ref{THM:Zalpha-cell-complex}, we prove that $Z_\alpha$ is homeomorphic to a ball and has a cell structure reflecting the set of $\alpha$-trees.
This definition of $Z_\alpha$ works when $n_\alpha\geq 3$. Note, that for $\alpha=(\oo\oo)$, there is exactly one $(\oo\oo)$-tree, which is already maximally expanded. We thus define $Z_{(\oo\oo)}:=\{*\}$ to be a one-point set.
\end{defn}
We can easily determine $\Delta_Q$ as follows. 
\begin{lem}\label{LEM:DeltaQ-easy}
\begin{equation}\label{EQN:DeltaQ}
\Delta_Q=\left\{\sum_{i=1}^{k_\alpha} x_ie_i\in \R^{k_\alpha}: x_1+\dots +x_{k_{\alpha}}=area(Q), \text{  and }x_i\geq 0 \text{ for all }i \right\}
\end{equation}
\end{lem}
\begin{proof}
To check the inclusion ``$\subseteq$'' note that all vectors $w_T$ from \eqref{EQN:wT} are in the right-hand side of \eqref{EQN:DeltaQ}, since $\bigcup_i Q_T(i)=Q$ with zero area intersections, and thus $\sum_i area(Q_T(i))=area(Q)$. For the other inclusion ``$\supseteq$'' it is enough to check that the vectors $area(Q)e_i$ are in $\Delta_Q$ for all $i=1,\dots, k_\alpha$, since $\Delta_Q$ is convex. To see this, let $S$ be any maximally expanded Stasheff-type tree. Then we can construct an essential spine $E_i$ which is $(j_1,\dots,j_{k_\alpha})$-compatible with $S$ by labeling the exterior edge at the $j_i$th position with an outgoing edge, while the outgoing edges at positions $j_1,\dots, j_{i-1},j_{i+1},\dots, j_{k_\alpha}$ are labeled with $\leftrightarrow$, and thus there is a unique flow to the outgoing edge at the $j_i$th position. Note that for $T_i=f^{-1}(S,E_i)$, the outgoing vertices at positions $j_1,\dots, j_{i-1},j_{i+1},\dots, j_{k_\alpha}$ are cut at their boundary line segments in $Q$, so that $i$th ``subpolygon'' of $Q$ is the whole polygon, $Q_{T_i}(i)=Q$. We thus obtain $w_{T_i}=area(Q)\cdot e_i\in \Delta_Q$, which is what we needed to check.
\end{proof}

\begin{ex}\label{EX:polyhedra} 
In the following examples, we fix some polygon $Q$. Figures~\ref{FIG:Associahedra234},~\ref{FIG:Pairahedra}, and~\ref{FIG:Higher-Poly} depict projections of the subspaces $Z_\alpha$ of high-dimensional Euclidean spaces onto their affine hulls.
\begin{enumerate}
\item\label{EX:polyhedra-associahedra}
When $\alpha=(\oo\ii\ii\dots\ii\ii)$ has exactly one outgoing label, $k_\alpha=1$, we get that $\Delta_Q\cong \{*\}$, so that $Z_\alpha\cong K_{n_\alpha-1}$ is precisely the associahedron; see Figure \ref{FIG:Associahedra234}.
\begin{figure}[h]
\[
 \includegraphics[scale=1]{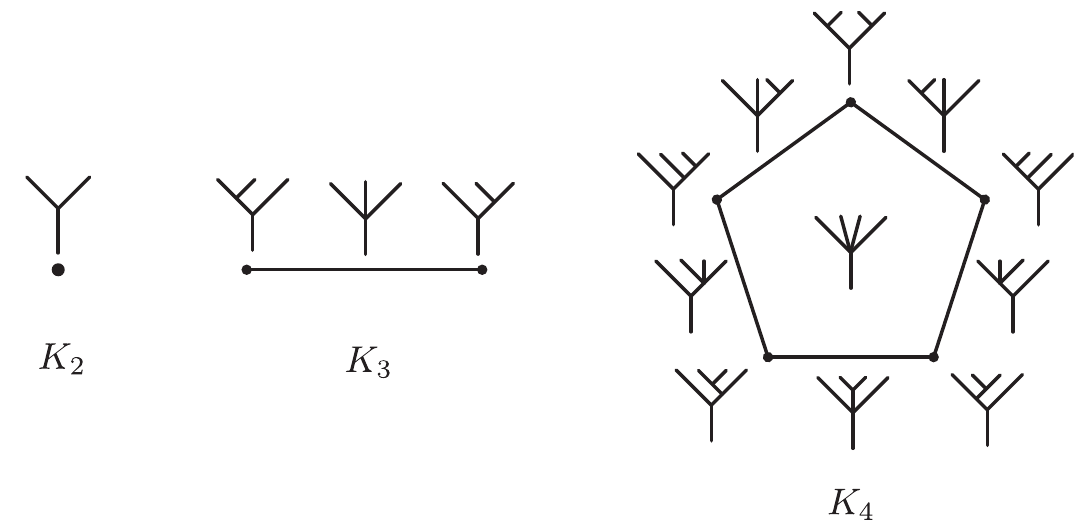}
\]
\caption{The spaces $Z_{(\oo\ii\ii)}\cong K_2$, $Z_{(\oo\ii\ii\ii)}\cong K_3$, and $Z_{(\oo\ii\ii\ii\ii)}\cong K_4$}\label{FIG:Associahedra234}
\end{figure}
\item\label{EX:polyhedra-OO}
When $\alpha=(\oo\ii\ii\dots\ii\ii\oo\ii\ii\dots\ii\ii)$ has exactly two outgoing labels, $k_\alpha=2$, we get that $\Delta_Q$ is an interval. Note that $\alpha$ is determined by exactly two numbers $\ell_1$ and $\ell_2$, which are the number of incoming edges between the two outgoing edges, $\alpha=(\oo\underbrace{\ii\ii\dots\ii\ii}_{\ell_1}\oo\underbrace{\ii\ii\dots\ii\ii}_{\ell_2})$. In this case $Z_\alpha$ is precisely the pairahedron as defined in \cite{T} for the two integers $\ell_1$ and $\ell_2$. In Figure \ref{FIG:Pairahedra} we display $Z_{(\oo\oo\ii)}$ (which is an interval $K_2\times \Delta^1$), $Z_{(\oo\ii\oo\ii)}$ (which is a hexagon that is a subdivision of $K_3\times \Delta^1$), and also $Z_{(\oo\ii\ii\oo\ii)}$ (which is a subdivision of $K_4\times \Delta^1$).
\begin{figure}[h]
\[
\includegraphics{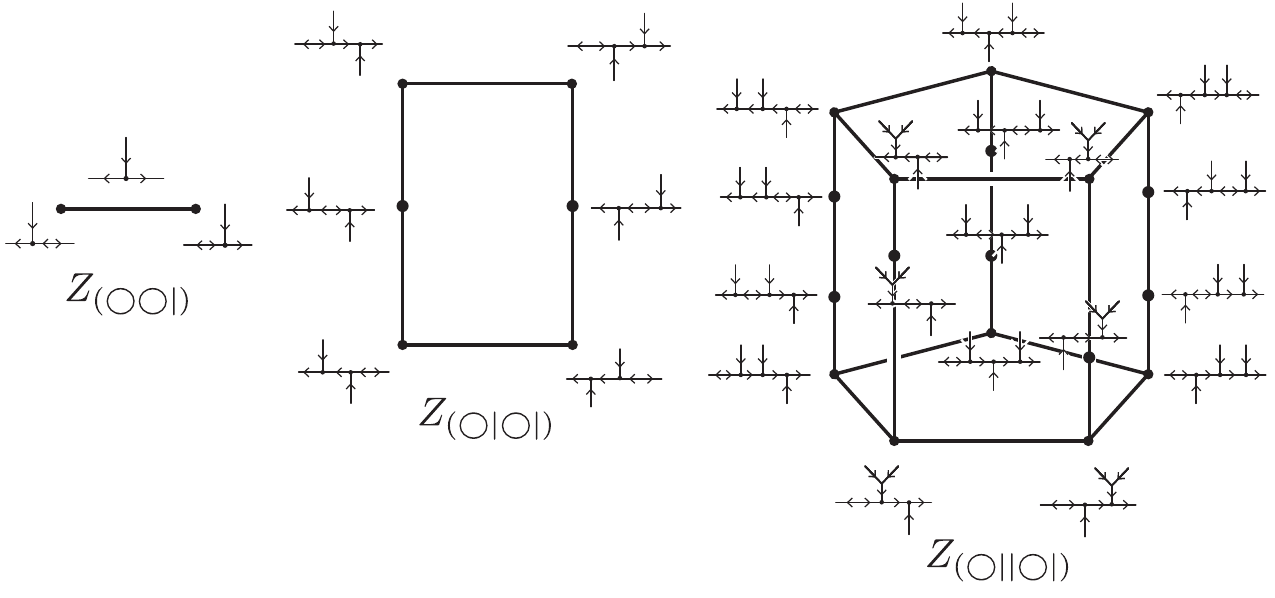} 
\]
\caption{ The spaces $Z_{(\oo\oo\ii)}$, $Z_{(\oo\ii\oo\ii)}$, and $Z_{(\oo\ii\ii\oo\ii)}$.}\label{FIG:Pairahedra}
\end{figure}
\item\label{EX:polyhedra-OOO}
Finally, we also display the spaces $Z_{(\oo\oo\oo)}\cong \Delta^2$ and $Z_{(\oo\oo\oo\ii)}$ (which is a subdivision of $K_3\times \Delta^2$) in Figure \ref{FIG:Higher-Poly}.
\begin{figure}[h]
\[
 \includegraphics[scale=1]{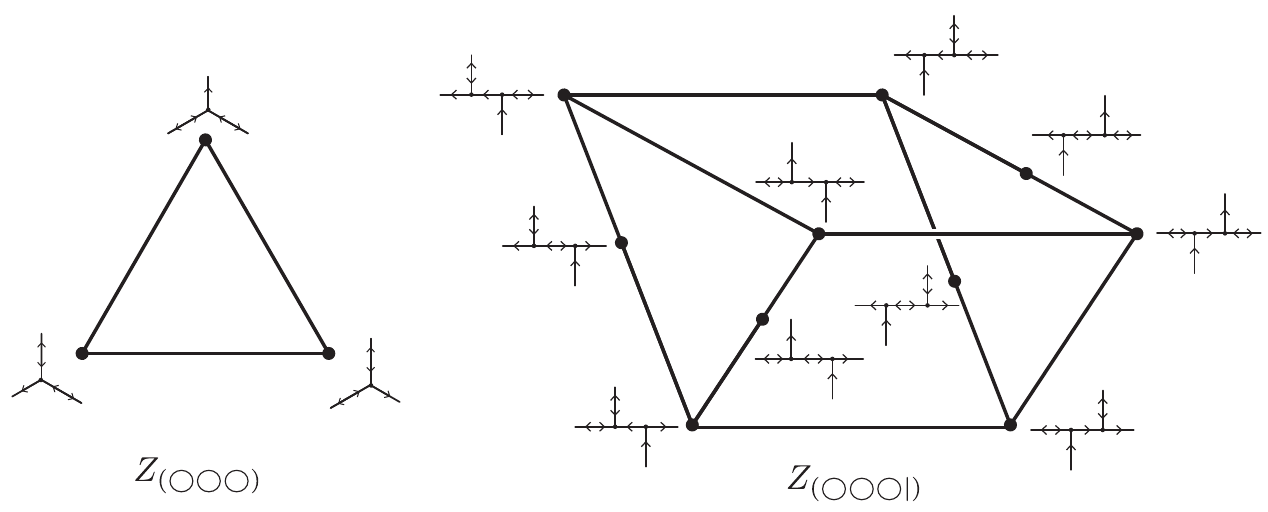}
\]
\caption{The spaces $Z_{(\oo\oo\oo)}\cong \Delta^2$ and $Z_{(\oo\oo\oo\ii)}$}
\label{FIG:Higher-Poly}
\end{figure}
\end{enumerate}
\end{ex}

We claim that $Z_\alpha$ 
may be given the structure of a cell complex so that each cell is labeled by an $\alpha$-tree $T$. We now define the cell $Z_T\subseteq Z_\alpha$ corresponding to the  $\alpha$-tree $T$.
\begin{defn}\label{DEF:KT-DeltaT-ZT}
Let $T\in \mathcal T_\alpha$ be an $\alpha$-tree which is not necessarily maximally expanded. Let $max(T)\subseteq \mathcal T_\alpha$ be set of all edge expansions of $T$ which are maximally expanded $\alpha$-trees.
\begin{eqnarray*}
K_T&:=& conv(\{v_t: \exists T'\in max(T)\text{, and } t\text{  is the triangulation of $Q$} \\
&& \hspace{.67in} \text{corresponding to the Stasheff-type $S_{T'}$}\}),\\
\Delta_T&:=& conv(\{w_{T'}: T' \in max(T)\}), \\
Z_T&:=& conv(\{(v_t,w_{T'}): \exists T'\in max(T)\text{, and } t\text{  is the triangulation of $Q$} \\
&& \hspace{1.05in} \text{corresponding to the Stasheff-type $S_{T'}$}\}).
\end{eqnarray*}
When $T$ is a corolla ($T$ has one internal vertex) then $K_T$ is equal to $K_Q$, $\Delta_T$ is equal to $\Delta_Q$, and $Z_T$ is equal to $Z_\alpha$.
In general,
it is clear that $Z_T$ is contained in $K_T \times \Delta_T$ and 
we will see below (Corollary \ref{COR:ZT=KTxDT}), that $Z_T$ is in fact equal to $K_T\times \Delta_T$.
\end{defn}

It will be convenient for us to work with an auxiliary space $\Lambda_T$ in $\R^{k_\alpha}$ which we define here.
Let $T$ be an $\alpha$-tree, and let $\e$ be any edge in $T$. The edge $\e$ of $T$ corresponds to a diagonal of the $n_\alpha$-gon $Q$ (which may possibly be a line segment of the boundary). This yields two separate polygons $Q'$ and $Q''$ (one of which may be a line segment) with $Q'\cup Q''=Q$ and $Q'\cap Q''=$diagonal corresponding to $\e$.
Furthermore, the set of outgoing edges of $T$ are subdivided into two subsets by $\e$, whose corresponding basis vectors in $\R^{k_\alpha}$ are given by $\{e_{i'_1},\dots,e_{i'_{k'}}\}$ and $\{e_{i''_1},\dots,e_{i''_{k''}}\}$ with $\{e_{i'_1},\dots,e_{i'_{k'}}\}\cup\{e_{i''_1},\dots, e_{i''_{k''}}\}=\{e_{1},\dots,e_{k_\alpha}\}$ and $\{e_{i'_1},\dots,e_{i'_{k'}}\}\cap\{e_{i''_1},\dots,e_{i''_{k''}}\}=\varnothing$.
Define the subspace $\Lambda_T$ of $\R^{k_\alpha}$ as follows:
\begin{multline*}
\Lambda_T:= \Big\{w=\sum_{i=1}^{k_\alpha} x_ie_i\in \R^{k_\alpha}:  x_1+\dots +x_{k_\alpha}=area(Q), x_i\geq 0, \text{ and for each}\\
\hspace{.0in}\text{edge $\e$ of $T$ labeled by $\leftrightarrow$ we have $x_{i'_1}+\dots+x_{i'_{k'}}= area(Q')$, and for each}\\
\hspace{.0in}\text{edge $\e$ of $T$ directed from $Q'$ to $Q''$ we have $x_{i'_1}+\dots+x_{i'_{k'}}\geq area(Q')$}\Big\}.
\end{multline*}
Note that if 
$\e$ is directed from $Q'$ to $Q''$ 
and $\{e_{i''_1},\dots, e_{i''_{k''}}\}$ is nonempty, then
for $w=\sum_{i=1}^{k_\alpha} x_ie_i$ in $\Lambda_T$,
we automatically have $x_{i''_1}+\dots+x_{i''_{k''}}\leq area(Q'')$.

The next two lemmas will show that $\Lambda_T$ and $\Delta_T$ are in fact equal.

\begin{lem}\label{LEM:inequalities}
Let $T$ be an $\alpha$-tree.
The space $\Delta_T\subseteq \R^{k_\alpha}$ is contained in $\Lambda_T$.
\end{lem}

\begin{proof}
Since $\Lambda_T$ is convex, it is enough to check that for any maximal expansion $T'$ of $T$, $w_{T'}$ is in $\Lambda_T$.
By Lemma~\ref{LEM:DeltaQ-easy}, $w_{T'}$ satisfies the first two conditions.

If $\e$ in $T$ has label $\leftrightarrow$, then $Q$ gets divided into $Q'$ and $Q''$, which
in $T'$
get further divided into regions with areas $x_{i'_1},\dots ,x_{i'_{k'}}$.
 Therefore, the sum of the coordinates $x_{i'_1}+\dots +x_{i'_{k'}}$ is equal to $area(Q')$.
See Figure \ref{FIG:Q-subdivided-Q'Q''}.
\begin{figure}[h]
\[
 \includegraphics[scale=0.9]{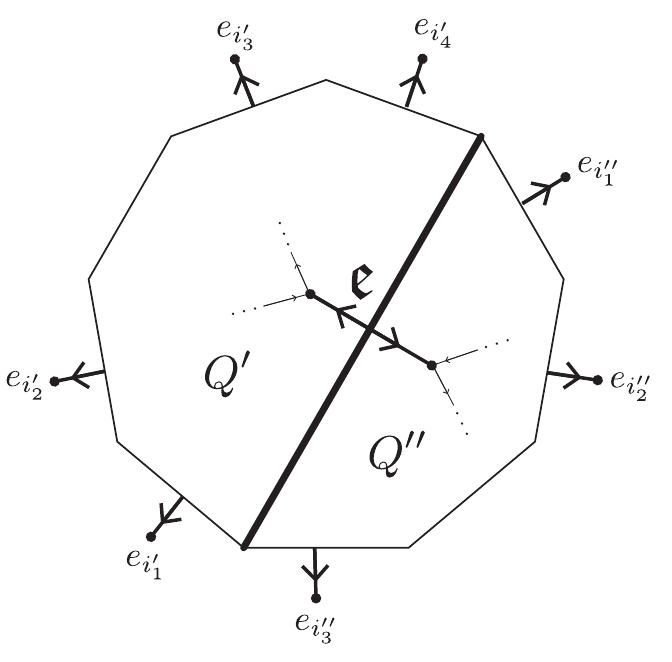}
\]
\caption{$Q$ subdivided by $Q'$ and $Q''$ via ``$\leftrightarrow$'' along the edge $\e$}\label{FIG:Q-subdivided-Q'Q''}
\end{figure}

It follows by a straightforward induction, that any convex polygon with $k$ outgoing edges that is divided into $k$ subpolygons has exactly $k-1$ dividing edges. Thus, if there is a direction from $Q''$ to $Q'$, then $Q'$ is subdivided into $k'$ many subpolygons, while $Q''$ is divided into $k''+1$ many subpolygons; see Figure \ref{FIG:Q-subdivided-Q'Q''-arrow}.
\begin{figure}[h]
\[
 \includegraphics[scale=0.9]{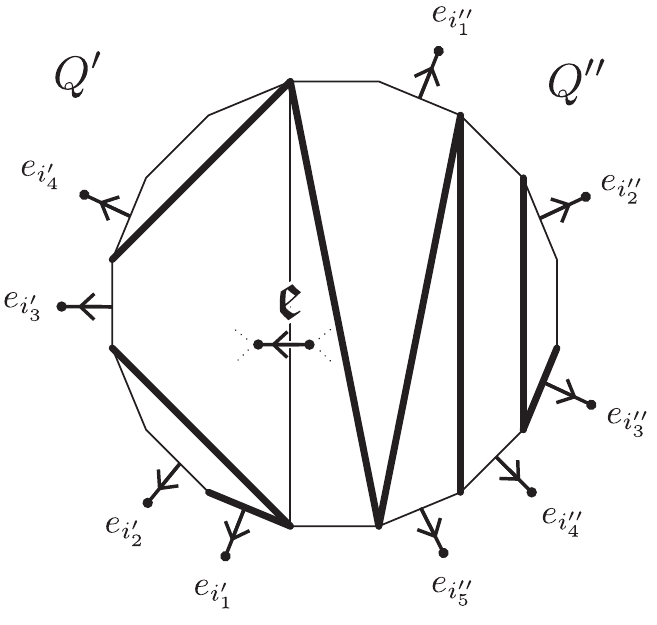}
\]
\caption{
$Q$ subdivided by $Q'$ and $Q''$ via an arrow from $Q''$ to $Q'$ along the edge $\e$}\label{FIG:Q-subdivided-Q'Q''-arrow}
\end{figure}

The subpolygon of $Q''$ corresponding to the outgoing edge $\e$ may thus provide an additional area from $Q''$ to be added to the outgoing edges of $Q'$. This shows that $x_{i'_1}+\dots+x_{i'_{k'}}\geq area(Q')$. 
\end{proof}

Notice that the map that takes $max(T)$ to $\R^{k_\alpha}$ by sending the maximally expanded tree $T'$ to $w_{T'}$ is not injective in general.
In particular, let $T^\circ$ be an expansion of $T$ obtained by replacing the labels on all but one outgoing edge at each vertex $\ve$ by $\leftrightarrow$.
Let $T_1'$ and $T_2'$ in $max(T)$ be expansions of $T^\circ$.
Then $T_1'$ and $T_2'$ yield the same decomposition of $Q$ because none of the the edges of $T_1'$ or $T_2'$ that do not appear in $T^\circ$ have the label $\leftrightarrow$.
Therefore the vectors 
$w_{T_1'}$ and $w_{T_2'}$ 
are equal.
We denote such vectors by $w_{T^\circ}$.
We will use the convex hull of all such $w_{T^\circ}$ as another auxiliary space: $W_T:=conv(\{w_{T^\circ} \;|\; T^\circ \text{ is obtained from } T \text{ as above}\})$.

We use a product of simplices to organize the set of such expansions $T^\circ$ as follows.
First note that, if the vertices of $E_T$ are $\ve_1, \dots, \ve_p$,
the set of expansions $T^\circ$ obtained from $T$ by changing edge labels as above is in bijective correspondence with the $0$-cells of the cell complex $\Delta^{o_{\ve_1}-1}\times \dots \times \Delta^{o_{\ve_p}-1}$.
Namely, for the vertex $\ve_i$ of $T$, if all but the $j$-th outgoing edge is relabeled by $\leftrightarrow$ in $T^\circ$, this corresponds to the $0$-cell $(0, \dots, 1, \dots, 0)$ of the factor $\Delta^{o_{\ve_i}-1}$ with $1$ in the $j$-th position.

Below, we will define an extension $h_T: \Delta^{o_{\ve_1}-1}\times \dots \times \Delta^{o_{\ve_p}-1} \to \Delta_T$ of the map that sends the $0$-cells of $\Delta^{o_{\ve_1}-1}\times \dots \times \Delta^{o_{\ve_p}-1}$ to the corresponding coordinates $w_{T^\circ}$.
This map $h_T$ will turn out to be a homeomorphism.

\begin{lem}\label{LEM:Delta_T}
The spaces $\Delta_T$, $\Lambda_T$, and $W_T$ are all equal:
\[ \Delta_T=\Lambda_T=W_T \]
Furthermore, there exists a homeomorphism $h_T: \Delta^{o_{\ve_1}-1}\times \dots \times \Delta^{o_{\ve_p}-1} \to \Delta_T$ which restricts to the map above that sends the $0$-cells of $\Delta^{o_{\ve_1}-1}\times \dots \times \Delta^{o_{\ve_p}-1}$ to the corresponding coordinates $w_{T^\circ}$.
\end{lem}
\begin{proof}
 We will proceed by checking the following four facts.
\begin{enumerate}
\item\label{ITEM:def-of-h-T}
There is a continuous map $h_T:(\Delta^{o_{\ve_1}-1}\times \dots \times \Delta^{o_{\ve_p}-1})\to \Lambda_T$.
\item\label{ITEM:h-T-is-invertible}
The map $h_T$ has an inverse map $h^{-1}_T:\Lambda_T\to (\Delta^{o_{\ve_1}-1}\times \dots \times \Delta^{o_{\ve_p}-1})$.
\item\label{ITEM:corners-in-Delta}
The map $h_T$ maps each $0$-cell of $(\Delta^{o_{\ve_1}-1}\times \dots \times \Delta^{o_{\ve_p}-1})$ into $\Delta_T$.
\item\label{ITEM:image-h-t-conv-span}
The image of $h_T$ lies in the convex span of the $h_T$ applied to the $0$-cells of $(\Delta^{o_{\ve_1}-1}\times \dots \times \Delta^{o_{\ve_p}-1})$.
\end{enumerate}
From these four facts, all the remaining claims of the Lemma follow, since $h_T$ is a homeomorphism by \eqref{ITEM:def-of-h-T} and \eqref{ITEM:h-T-is-invertible}, and thus
\[
\Lambda_T\stackrel{\eqref{ITEM:def-of-h-T},\eqref{ITEM:h-T-is-invertible}}=image(h_T)\stackrel{\eqref{ITEM:image-h-t-conv-span}}\subseteq W_T \stackrel{\eqref{ITEM:corners-in-Delta}}\subseteq\Delta_T.
\]
By Lemma~\ref{LEM:inequalities}, we have $\Delta_T \subseteq \Lambda_T$, so each of these containments is non-proper, which completes the proof.

$\bullet$ To check \eqref{ITEM:def-of-h-T}, for each internal vertex $\ve_i$ of $E_T$ (where $i\in\{1,\dots, p\}$) with $o_{\ve_i}$ outgoing edges, we parametrize the standard simplex $\Delta^{o_{\ve_i}-1}$ with the coordinates
\[
\Delta^{o_{\ve_i}-1}=\{(t_{i,1},\dots,t_{i,o_{\ve_i}}): t_{i,1}+\dots+t_{i,o_{\ve_i}}=1, t_{i,1}\geq 0, \dots, t_{i,o_{\ve_i}}\geq 0 \}.
\]
Using the edges of $T$, the polygon $Q$ is subdivided into, say, $p$ subpolygons, by diagonals in $Q$ or by line segments at the boundary of $Q$. (Compare this with Definition \ref{DEF:w_T}, but now using all of $T$ instead of just the essential spine $E_T$.) Line segments at the boundary of $Q$ labeled with ``$\leftrightarrow$'' give subpolygons of $Q$ which are degenerate, and thus have zero area; see e.g. $Q_{11}$ and $Q_{12}$ in Figure \ref{FIG:Q-subdivided-for-ts}.
\begin{figure}[h]
\[
 \includegraphics[scale=1.2]{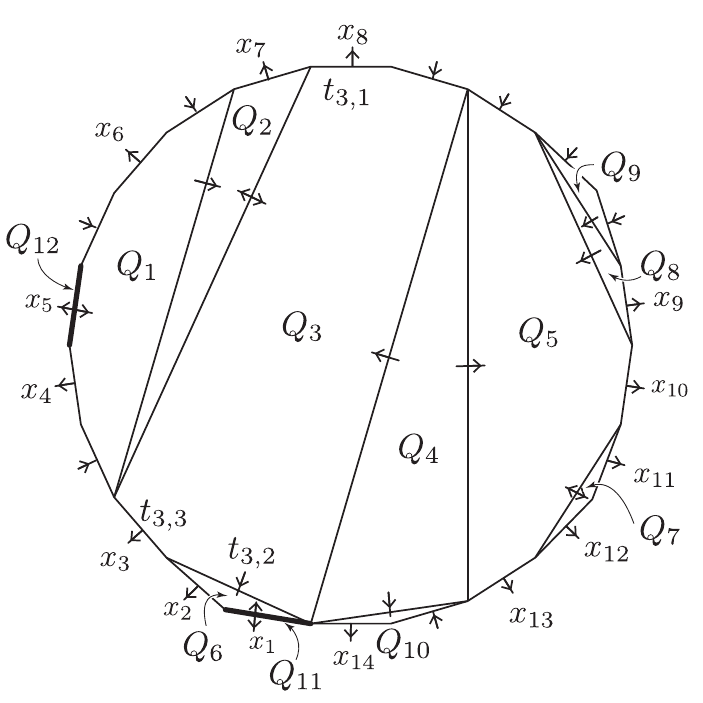}
\]
\caption{$Q$ subdivided into $Q_1,\dots,Q_{12}$; outgoing coordinates are marked by $x_1,\dots, x_{14}$; coordinates of $\Delta^{\ve_3-1}$ (as an example) are marked by $t_{3,1}$, $t_{3,2}$, $t_{3,3}$}\label{FIG:Q-subdivided-for-ts}
\end{figure}

Since $T$ is an $\alpha$-tree, there is a partial order on these subpolygons together with the outgoing coordinates given by the directions on $T$. For example, $Q$ from Figure \ref{FIG:Q-subdivided-for-ts} gives the following partial order:
\[\scalebox{0.57}{\xymatrix{
&&&&&& Q_4\ar[dll]\ar[ddrr]\ar[dd] && Q_9\ar[d] \\
&Q_1\ar[ddl]\ar[dd]\ar[dr]&&& Q_3\ar[dl]\ar[dd]\ar[ddr] &&&& Q_8\ar[d]\ar[ddr]\\
 &  & Q_2\ar[d] &   Q_6\ar[d]& & & Q_{10}\ar[d] & & Q_5\ar[dl]\ar[d] &&&  Q_7\ar[dl]\ar[dr] && Q_{12}\ar[d] & Q_{11}\ar[d]\\
x_{4} & x_{6} & x_{7} & x_{2} & x_{3} & x_{8} & x_{14}  & x_{10} & x_{13} & x_{9} & x_{11} & & x_{12} & x_{5} & x_{1} }}\]
Note that there are no cycles and that the minimal elements are the coordinates $x_j$. We define $h_T$ by distributing areas of the polygons $Q_i$ to the areas of the polygons appearing right below $Q_i$ in this partial order, respectively to the coordinates $x_j$ appearing below $Q_i$ in the partial order. In fact, for $(t_{i,1},\dots,t_{i,{o_{\ve_i}}})\in \Delta^{o_{\ve_i}-1}$, if $t_{i,j}$ is the coordinate for an edge from $Q_i$ to $Q_{i'}$, then we add $t_{i,j}\cdot area(Q_i)$ to the area of $Q_{i'}$. We will denote the area of $Q_{i'}$ with the added part by $area^+(Q_{i'})$. Since $\sum_{j=1}^{o_{\ve_i}} t_{i,j}=1$, all of the area of $Q_i$ gets distributed to the next polygon or coordinate $x_j$. Continuing in this way inductively with the newly adjusted ``$area^+$''s for our polygons, we arrive at our final output numbers $x_1,\dots, x_{k_\alpha}$ for which $\sum_{i=1}^{k_\alpha} x_i e_i$ gives the output under the map $h_T$. Note, that this map lands in $\Lambda_T$, since for each symbol $\leftrightarrow$, the areas are completely separated, while for each direction arrow, say from $Q''$ to $Q'$ some of the area from $Q''$ may be distributed to $Q'$, giving the wanted equalities and inequalities stated in the definition of $\Lambda_T$. Furthermore, it is clear that $h_T$ is continuous, as it is given by additions and multiplications.

$\bullet$ To check \eqref{ITEM:h-T-is-invertible}, we explicitly describe the inverse map of $h_T$. Starting from $\sum_{i=1}^{k_\alpha}x_i e_i\in \Lambda_T$, let $Q'$ be a subpolygon which has no outward pointing direction to any other subpolygon $Q_i$; for example $Q_2, Q_5, Q_6, Q_7, Q_{10}, Q_{11}$, and $Q_{12},$ in Figure \ref{FIG:Q-subdivided-for-ts}. Note that cutting along $Q'$ subdivides $Q$ into subpolygons; call them $R_1,\dots, R_q$ (for example in Figure \ref{FIG:Q-subdivided-for-ts}, take $Q'=Q_5$ and $R_1=Q_7$, $R_2=Q_8\cup Q_9$, $R_3=Q_1\cup Q_2\cup Q_3\cup Q_4\cup Q_5 \cup Q_6\cup Q_{10}\cup Q_{11}\cup Q_{12}$). Denote those coordinates that receive outgoing directions from $Q'$ by $x_{i'_1},\dots, x_{i'_{k'}}$, while the outgoing variables from $R_j$ for $j=1,\dots, q$, are denoted by $x_{i''_{j,1}},\dots, x_{i''_{j,k_j}}$. We define coordinates $(t'_1,\dots,t'_{k'})\in \Delta^{k'-1}$ by setting $t'_j:=\frac{x_{i'_j}}{x_{i'_1}+\dots+x_{i'_{k'}}}$, so that clearly $t'_1+\dots +t'_{k'}=1$. 

We claim that we can repeat this procedure for each subpolygon $R_j$. First, note that $R_j$ has exactly the outgoing edges with coordinates $x_{i''_{j,1}},\dots, x_{i''_{j,k_j}}$ together with a new outgoing edge that was pointed toward $Q'$. We want to associate a number $a_j$ to this new outgoing edge so that we can repeat the above procedure for $R_j$. From the (in-)equalities defining $\Lambda_T$,  we see that for each label ``$\leftrightarrow$'' between $Q'$ and $R_j$ there is an equality $x_{i''_{j,1}}+\dots+x_{i''_{j,k_j}}= area(R_j)$, while for each arrow from $R_j$ incoming into $Q'$ there is an inequality $x_{i''_{j,1}}+\dots+x_{i''_{j,k_j}}\leq area(R_j)$. Let $a_j:=area(R_j)-(x_{i''_{j,1}}+\dots+x_{i''_{j,k_j}})\geq 0$. (Informally, $a_j$ is the ``excess area'' that gets transferred from $R_j$ to $Q'$.) Thus, the outward pointing edges of $R_j$ have a total number of $a_j+ x_{i''_{j,1}}+\dots+x_{i''_{j,k_j}}=area(R_j)$ associated with them. Furthermore, these numbers satisfy the inequalities required in $\Lambda_{T_j}$ where $T_j$ is the tree that corresponds to the polygon $R_j$. (This can be seen, because each edge in $T_j$ determines an (in-)equality, which can be expressed in two ways: one involving $a_j$ and one not involving $a_j$. The (in-)equalities not involving $a_j$ are the same as in $\Lambda_{T_j}$ and in $\Lambda_T$.) Thus, by induction, we can repeat this process and obtain coordinates in $\Delta^{o_{\ve_i}-1}$ for each internal vertex $\ve_i$ of $T$.

Finally we note that the above description is the inverse of $h_T$. To see this, starting from $w=\sum_{i=1}^{k_\alpha} x_i e_i\in \Lambda_T$, in the above notation using $Q', R_1,\dots, R_q$, we obtain the excess areas $a_j$ at each direction from $R_j$ to $Q'$. Now, to apply $h_T$, we need to assign to this the output coordinates $t_{i'_j}\cdot area^+(Q')=t_{i'_j}\cdot (area(Q')+\sum_{j=1}^q a_j)$, as stated in the definition of $h_T$ in \eqref{ITEM:def-of-h-T}.  According to the definition of $h_T^{-1}$, we have $t_{i'_j}=\frac{x_{i'_j}}{x_{i'_1}+\dots+x_{i'_{k'}}}$. Furthermore, since $area(Q)=\sum_{i=1}^{k_\alpha}x_i=(\sum_{j=1}^{k'}x_{i'_j})+\sum_{j=1}^q(\sum_{\ell=1}^{k_j} x_{i''_{j,\ell}})=(\sum_{j=1}^{k'}x_{i'_j})+\sum_{j=1}^q(area(R_j)-a_j)=(\sum_{j=1}^{k'}x_{i'_j})+area(Q)-area(Q')-\sum_{j=1}^q a_j$, it follows that $area(Q')+\sum_{j=1}^q a_j=x_{i'_1}+\dots+x_{i'_{k'}}$.  Thus, applying $h_T^{-1}$ composed with $h_T$ yields the coordinates $t_{i'_j}\cdot area^+(Q')=x_{i'_j}$, which gives $h_T(h^{-1}_T(w))=\sum_{i=1}^{k_\alpha} x_i e_i=w$.

Conversely, starting from $t_{i,j}$ in $\Delta^{o_{\ve_1}-1}\times \dots \times \Delta^{o_{\ve_p}-1}$, and applying $h_T$ to this, we obtain, the adjusted areas $area^+(Q')$, and from this the coordinates $x_{i'_j}=t_{i'_j}\cdot area^+(Q')$. Applying $h^{-1}_T$ to these gives $\frac{x_{i'_j}}{x_{i'_1}+\dots+x_{i'_{k'}}}=\frac{t_{i'_j}\cdot area^+(Q')}{(t_{i'_1}+\dots+ t_{i'_{k'}})\cdot area^+(Q')}=t_{i'_j}$. Thus $h_T^{-1}\circ h_T=id$ as well.

$\bullet$ Item \eqref{ITEM:corners-in-Delta} follows immediately since $w_{T^\circ}$ are coordinates corresponding to maximal expansions of $T$ and since $\Delta_T$ is convex.

$\bullet$ To check \eqref{ITEM:image-h-t-conv-span}, note that when fixing coordinates $t_{i,j}$ for all but one internal vertex $\ve_{i_0}$, the map $h_T$ becomes a map $\tilde h_T:\Delta^{o_{\ve_{i_0}}-1}\to \Lambda_T$, which is just an \emph{affine} map (given by distributing the area of $Q_{i_0}$ to the output coordinates $x_j$ and adding other fractional parts of areas to those). Thus, the image of such a $\tilde h_T$ is in the convex hull of the image of the $0$-cells of $\Delta^{o_{\ve_{i_0}}-1}$. Let $((t_{1,1},\dots,t_{1,o_{\ve_1}}),\dots,(t_{p,1},\dots, t_{p,o_{\ve_p}}))\in (\Delta^{o_{\ve_1}-1}\times \dots \times \Delta^{o_{\ve_p}-1})$ be the coordinates of any element in the domain of $h_T$. By \eqref{ITEM:corners-in-Delta}, we know that the images of $0$-cells $((0,\dots,1,\dots,0),\dots,(0,\dots,1,\dots,0))\in (\Delta^{o_{\ve_1}-1}\times \dots \times \Delta^{o_{\ve_p}-1})$ lie in $\Delta_T$. Fixing coordinates for $\ve_2,\dots ,\ve_p$, and letting $h_T$ depend only on $\Delta^{o_{\ve_1}-1}$, we see that the image of \[((t_{1,1},\dots,t_{1,o_{\ve_1}}),(0,\dots,1,\dots,0),\dots,(0,\dots,1,\dots,0))\] also lies in the convex set $\Delta_T$. Now, fixing $(t_{1,1},\dots,t_{1,o_{\ve_1}})\in \Delta^{o_{\ve_1}-1}$ as well as any $0$-cell in $\Delta^{o_{\ve_3}-1}, \dots, \Delta^{o_{\ve_p}-1}$, and letting $h_T$ only vary over $\Delta^{o_{\ve_2}-1}$, we see that the image of
\[
((t_{1,1},\dots,t_{1,o_{\ve_1}}),(t_{2,1},\dots,t_{2,o_{\ve_2}}),(0,\dots,1,\dots,0),\dots,(0,\dots,1,\dots,0))
\]
also lies in $\Delta_T$. Continuing this way, we see that the element we started with also maps to $\Delta_T$, i.e. $h_T((t_{1,1},\dots,t_{1,o_{\ve_1}}),\dots,(t_{p,1},\dots, t_{p,o_{\ve_p}}))\in \Delta_T$.
\end{proof}
\begin{cor}\label{COR:ZT=KTxDT}
The space
$Z_T$ is equal to the space
$K_T\times \Delta_T$, which is homeomorphic to a closed ball $B^d$ of dimension
\begin{eqnarray*}
d&=&(n_\alpha-3)-(\text{number of internal edges of }S_T)\\
&&+\sum_{\tiny \begin{array}{cc} \text{$\ve$: $\ve$ is internal}\\ \text{vertex of $E_T$}\end{array}}\Big((\text{number of outgoing edges of $\ve$})-1\Big).
\end{eqnarray*}
\end{cor}
\begin{proof}
Clearly, $Z_T\subseteq K_T\times \Delta_T$. Conversely, an element in $K_T\times \Delta_T$ is
in the convex hull of  tuples $(v_t, w_{T''})$, where $v_t$ corresponds to a maximal tree $T'$ and $w_{T''}$ corresponds to a maximal tree $T''$. 
Since $Z_T$ is convex, it is enough to check that each such $(v_t, w_{T''})$  is in $Z_T$.
Since $\Delta_T$ is equal to $W_T$, the convex hull of $\{w_{T^\circ} \}$, it is enough to check that each $(v_t, w_{T^\circ})$ is in $Z_T$.

We claim that there exists a maximal expansion $T'''$ of $T^\circ$ whose underlying Stasheff tree $S_{T'''}$ is equal to $S_{T'}$.
To construct such a tree $T'''$ we start with $S_{T'}$ and change the labels of (some of) its edges.
Notice first that the underlying Stasheff trees $S_T$ and $S_{T^\circ}$ are equal since $T^\circ$ is obtained from $T$ purely by changing labels of some edges, and so $S_{T'}$ is an expansion of $S_{T^\circ}$.
To construct $T'''$, label edges of $S_{T'}$ as follows.
Edges of $S_{T'}$ that correspond to edges of $S_{T^\circ}$ are given the same labels as in $T^\circ$.
Edges of $S_{T'}$ that do not  correspond to edges of $S_{T^\circ}$ are given the unique directions so that $T'''$  satisfies the conditions of Definition~\ref{DEF:directed-trees}.

Since $T'''$ is a maximal expansion of $T^\circ$, $w_{T'''}$ is equal to $w_{T^\circ}$.
And since $S_{T'''}$ is equal to $S_{T'}$ they have the same vector $v_t$.
Therefore, $(v_t, w_{T^\circ})$ is equal to $(v_t, w_{T'''})$ which is in $Z_T$.

The dimension formula follows from the homeomorphism $h_T:(\Delta^{o_{\ve_1}-1}\times \dots \times \Delta^{o_{\ve_p}-1})\to \Delta_T$ from the proof of Lemma~\ref{LEM:Delta_T}, since $\Delta^{o_{\ve_1}-1}\times \dots \times \Delta^{o_{\ve_p}-1}$ has dimension $(o_{\ve_1}-1)+\dots +(o_{\ve_p}-1)$, while the associahedron $K_T$ is of dimension $(n_\alpha-3)-($number of internal edges of $S_T)$.
\end{proof}

To show that the cells $Z_T$ give $Z_\alpha$ the structure of a cell complex, we first analyze how the collection of spaces $Z_T$ sit inside the space $Z_\alpha$.
Recall that the \emph{relative interior} of a set $\mathcal S\subseteq \R^m$ is the interior of $\mathcal S$ as sitting inside its affine hull, and there is a similar version for the \emph{relative boundary} of $\mathcal S$.
\begin{lem}\label{LEM:Z-cell-properties} \quad
\begin{enumerate}
\item\label{ITEM:Z-disjoint-union}
$Z_\alpha$ is the disjoint union of the relative interiors of $Z_T$ over all trees in $\mathcal T_\alpha$:
\[
Z_\alpha
=\coprod
_{T\in \mathcal T_\alpha} \ri(Z_T)  
=\coprod
_{T\in \mathcal T_\alpha} \ri(K_T)\times \ri(\Delta_T). 
\]
\item\label{ITEM:Z-boundary}
The relative boundary of $Z_T$ is equal to the union of all $Z_{T'}$ where $T'$ is an edge expansion of $T$:
\[
\rbd(Z_T)= 
\bigcup_{\tiny \begin{array}{cc} \text{$T'\in \mathcal T_\alpha$: $T'$ is edge}\\ \text{expansion of $T$}\end{array}} Z_{T'}
\]
\end{enumerate}
\end{lem}
\begin{proof}
For item \eqref{ITEM:Z-disjoint-union}, we need to show that the $\ri(K_T)\times \ri(\Delta_T)$ are all disjoint and that $Z_\alpha=\bigcup_{T\in \mathcal T_\alpha} \ri(Z_T)$. For the disjoint property, let $T$ and $T'$ be any two $\alpha$-trees and assume that there is an intersection, i.e. $(\ri(K_T)\times \ri(\Delta_T))\cap (\ri(K_{T'})\times \ri(\Delta_{T'}))\neq \varnothing$. Then, $\ri(K_T)\cap \ri(K_{T'})\neq \varnothing$ and $\ri(\Delta_T)\cap \ri(\Delta_{T'})\neq \varnothing$. The first implies that $T$ and $T'$ must have the same Stasheff tree $S_T=S_{T'}$, since this must certainly be true for the associahedra. Now, by Lemma \ref{LEM:Delta_T}, $\Delta_T$ and $\Delta_T'$ is given by some equalities and inequalities. Now, in the relative interior, the inequalities must necessarily be strict, so that $\ri(\Delta_T)$ and $\ri(\Delta_{T'})$ have an intersection only if all edges have the same direction associated with them. Thus, $E_T=E_{T'}$, so that $T=T'$. This show that $\ri(Z_T)$ and $\ri(Z_{T'})$ are disjoint when $T\neq T'$.

To show the union $Z_\alpha=\bigcup_{T\in \mathcal T_\alpha} \ri(Z_T)$ the only non-trivial part is that $Z_\alpha\subseteq \bigcup_{T\in \mathcal T_\alpha} \ri(Z_T)$. Consider an element in $Z_\alpha=K_Q\times \Delta_Q$, call it $(v,w)\in K_Q\times \Delta_Q$. Since $K_Q$ is the union of the relative interiors of $K_T$, there is a Stasheff tree $S$, such that $v\in \ri(K_S)$. We need to find labels (i.e. directions or $\leftrightarrow$) on the edges of $S$ so that we obtain an essential spine $E$ compatible with $S$ and then obtain a tree $T$ with $S_T=S$, and with $w\in \Delta_T$. Let $E$ be the tree obtained from $S$ by connecting all the outgoing edges given according to the labels from $\alpha$. We label the edges of $E$ by placing directions or $\leftrightarrow$ according to the description provided by Lemma \ref{LEM:inequalities}. If $w=\sum x_i e_i\in \Delta_Q$, and $\e$ is an edge in $S$, subdividing $Q$ into $Q'$ and $Q''$, and subdividing the set of basis vectors of the space of outgoing edges $\R^{k_\alpha}$ into $\{e_{i'_1},\dots,e_{i'_{k'}}\}$ and $\{e_{i''_1},\dots,e_{i''_{k''}}\}$, then we place $\leftrightarrow$ if $x_{i'_1}+\dots+x_{i'_{k'}}=area(Q')$, we place an arrow from $Q''$ to $Q'$ if $x_{i'_1}+\dots+x_{i'_{k'}}> area(Q')$, and we place an arrow from $Q'$ to $Q''$ if $x_{i''_1}+\dots+x_{i''_{k''}}> area(Q'')$. We claim that these directions make $E$ into an essential spine according to the Definition \ref{DEF:essential-spine}. Item \eqref{ITEM:choice-alpha} in Definition \ref{DEF:essential-spine} is immediate by the choice of $E$. Next, item \eqref{ITEM:label-edges} follows since each $x_i\geq 0$ (see Lemma \ref{LEM:DeltaQ-easy}) and thus the external edges are labeled outgoing or with the symbol $\leftrightarrow$ (when $x_i=0$). To check \eqref{ITEM:internal-one-outgoing}, we need to see that every internal vertex has at least one outgoing edge. In fact, if not, then there is a subpolygon, call it $Q'$, so that all arrows are incoming to $Q'$. Call the other remaining subpolygons $R_1,\dots R_p$, where $R_j$ has outgoing edges whose basis vectors are $\{e_{i_{j,1}},\dots, e_{i_{j,k_j}}\}$ for $j=1,\dots ,p$; see Figure \ref{FIG:Q-subdivided}. 

\begin{figure}[h]
\[
 \includegraphics[scale=1]{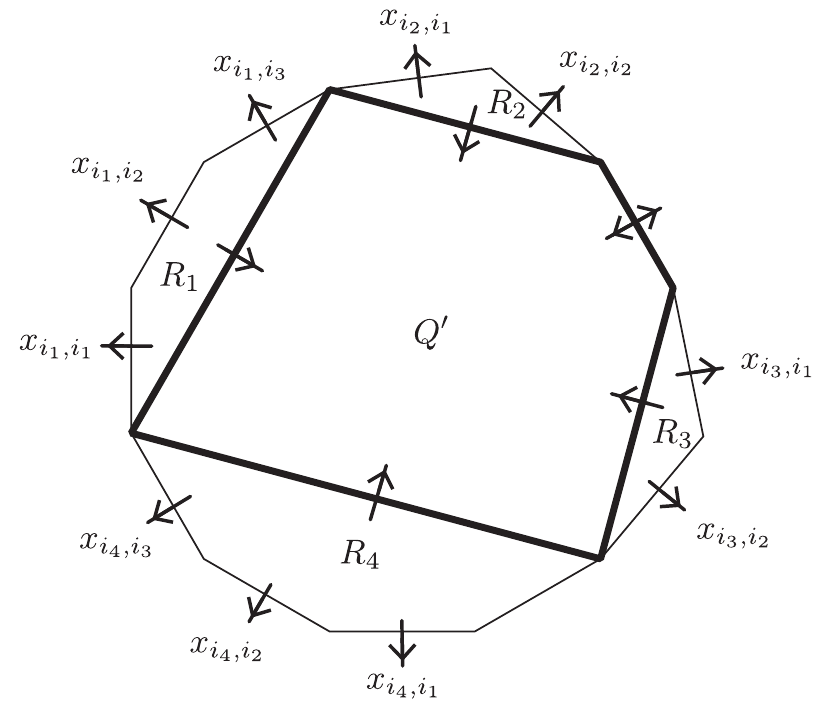}
\]
\caption{$Q$ subdivided by $Q'$ and $R_1,\dots,R_p$}\label{FIG:Q-subdivided}
\end{figure}

Since all the edges are incoming into $Q'$, the coordinates $x_{i_{j,1}}+\dots+x_{i_{j,k_j}}$ are less than or equal to $area(R_j)$ for each $j=1,\dots, p$. Adding all of these gives $\sum_i x_i\leq area(R_1)+\dots+area(R_p)=area(Q)-area(Q')$. Since $\sum_i x_i=area(Q)$, this means that $area(Q')\leq 0$, which is a contradiction. Thus, item \eqref{ITEM:internal-one-outgoing} from Definition \ref{DEF:essential-spine} is also satisfied, and we have an essential spine $E$. By choice, $E$ is compatible with $S$, and the corresponding $\alpha$-tree $T=f^{-1}(S,E)$ satisfies $w\in \Delta_T$ by the inequalities for $\Delta_T$ in Lemma \ref{LEM:Delta_T}. Thus, $(v,w)\in \ri(K_T)\times \ri(\Delta_T)$, showing $Z_\alpha\subseteq \bigcup_{T\in \mathcal T_\alpha} \ri(Z_T)$.

We now prove item \eqref{ITEM:Z-boundary}, that the relative boundary $\rbd(Z_T)$ is equal to the union of the $Z_{T'}$ of the edge expansions $T'$ of $T$. Since $Z_T=K_T\times \Delta_T$, we can write $\rbd(Z_T)=(\rbd(K_T)\times \Delta_T)\cup (K_T\times \rbd(\Delta_T))$. First, we check the inclusion $\rbd(Z_T)\supseteq  \bigcup_{\tiny \begin{array}{cc} \text{$T'$ is edge ex-}\\ \text{pansion of $T$}\end{array}} Z_{T'}$. If $T'$ is an edge expansion of $T$, then the underlying Stasheff graph $S_{T'}$ is equal to $S_T$ or it is an edge expansion of $S_T$. In the case where $S_{T'}$ is an edge expansion of $S_T$, it is well known that $K_{T'}$ is in the relative boundary of $K_T$. Furthermore, $\Delta_{T'}\subseteq \Delta_T$ since every maximal expansion of $T'$ is also a maximal expansion of $T$, so that in this case $Z_{T'}=K_{T'}\times \Delta_{T'}\subseteq \rbd(K_T)\times \Delta_T\subseteq \rbd (Z_T)$. In the case when $S_{T'}=S_T$, and thus $K_{T'}=K_T$, at least one of the labels in $E_{T'}$ was changed from a direction to $\leftrightarrow$. Using the description of $\Delta_T$ as the image of $h_T$ in Lemma \ref{LEM:Delta_T}, this shows that $\Delta_{T'}$ is in the relative boundary $\rbd(\Delta_T)$, so that again $Z_{T'}=K_{T'}\times \Delta_{T'}\subseteq K_T\times \rbd(\Delta_T)\subseteq \rbd (Z_T)$. Taking the union over all $T'$ shows that $\rbd(Z_T)\supseteq  \bigcup_{\tiny \begin{array}{cc} \text{$T'$ is edge ex-}\\ \text{pansion of $T$}\end{array}} Z_{T'}$.

Next, we check the other inclusion $\rbd(Z_T)=(\rbd(K_T)\times \Delta_T)\cup (K_T\times \rbd(\Delta_T))\subseteq \bigcup_{\tiny \begin{array}{cc} \text{$T'$ is edge ex-}\\ \text{pansion of $T$}\end{array}} Z_{T'}$.
In the case where we take the relative boundary of $K_T$, it is well known that the codimension one faces of the associahedra are given by edge expansions of $S_T$ by one edge, call the new edge $\e$. Let $T_0$ be the corresponding tree with the new edge $\e$, but without any direction label yet. In some cases $\e$ can only be labeled with a unique direction label, giving a new directed $\alpha$-tree $T'$. Since all maximal expansions of $T$ that include $\e$ also are maximal expansions of $T'$, we see that the subset of
$\rbd(K_T)\times \Delta_T$ corresponding to such an edge expansion is precisely $K_{T'}\times\Delta_{T'}$, labeled by this $T'$.

There is a second case, where $\e$ may be labeled with either direction as well as with $\leftrightarrow$. (This happens when the two internal vertices on either side of $\e$ already have at least one outgoing edge.) In this case, we obtain two expansions $T'_\rightarrow$ and $T'_\leftarrow$, with $\e$ labeled with either direction. According to Lemma \ref{LEM:Delta_T}, the direction of $\e$ induces one more inequality using an appropriate subpolygon $Q'$, i.e. $x_{i'_1}+\dots +x_{i'_{k'}}\geq area(Q')$ for $T'_\rightarrow$, or $x_{i'_1}+\dots +x_{i'_{k'}}\leq area(Q')$ for $T'_\leftarrow$, respectively. Moreover, by Lemma \ref{LEM:Delta_T}, $\Delta_{T'_\rightarrow}=\Delta_T\cap\{w=\sum x_i e_i: x_{i'_1}+\dots +x_{i'_{k'}}\geq area(Q')\}$ and $\Delta_{T'_\leftarrow}=\Delta_T\cap\{w=\sum x_i e_i: x_{i'_1}+\dots +x_{i'_{k'}}\leq area(Q')\}$, so that $\Delta_T=\Delta_{T'_\rightarrow}\cup \Delta_{T'_\leftarrow}$. Thus, 
the subset of 
 $\rbd(K_T)\times \Delta_T$ corresponding to such an edge expansion is precisely the union of the two spaces $K_{T'_\rightarrow}\times\Delta_{T'_\rightarrow}$ and $K_{T'_\leftarrow}\times\Delta_{T'_\leftarrow}$.

Finally, we consider the 
subset $K_T\times \rbd(\Delta_T)$ of $\rbd(Z_T)$.
Here, $T$ is expanded not by an expansion of $S_T$, but by replacing one of the labels of $E_T$ from a direction to $\leftrightarrow$. Then, define $T'$ by letting $S_{T'}=S_T$ and $E_{T'}$ be the new essential spine with label $\leftrightarrow$. Again, any maximal expansion of $T$ with this edge labeled $\leftrightarrow$ will also be a maximal expansion of $T'$ and so 
the subset of $K_T\times \rbd(\Delta_T)$ given by changing the label as described above
 is precisely $K_{T'}\times \Delta_{T'}$.

It follows that in all cases the relative boundary of $Z_T$ lies in the spaces $Z_{T'}$ where $T'$ is an edge expansion of $T$. This completes the proof of item \eqref{ITEM:Z-boundary}, and thus of the lemma.
\end{proof}

With all the prior work, we can now state and prove our main theorem.
\begin{thm}\label{THM:Zalpha-cell-complex}
The space $Z_\alpha$ has the structure of a cell complex where the cells are given by the subspaces $Z_T$ for $T$ in $\mathcal{T}_\alpha$.
This structure is a cellular subdivision of the product of an associahedron and a simplex $K_{n_\alpha-1}\times \Delta^{k_\alpha-1}$ in $\R^{n_\alpha}\times \R^{k_\alpha}$, each with their own natural cell complex structures.
\end{thm}
\begin{proof}
We define the $(d-1)$-skeleton of $Z_\alpha$ inductively by taking the union of all $Z_T$ over all $\alpha$-trees $T$, whose associated space $Z_T$ has dimension (from Corollary \ref{COR:ZT=KTxDT}) less than $d$. Now, if $T$ is so that $Z_T$ has dimension $d$, then by Lemma \ref{LEM:Z-cell-properties}\eqref{ITEM:Z-disjoint-union}, the relative interior of $Z_T$ is disjoint from the $(d-1)$-skeleton. By Lemma \ref{LEM:Z-cell-properties}\eqref{ITEM:Z-boundary}, the relative boundary of $Z_T$ lies in the cells $Z_{T'}$ of edge expansions $T'$ of $T$. From Definition \ref{DEF:formal-expansion-and-dim}, it follows, that an edge expansion $T'$ of $T$ by one edge exactly decreases the dimension of $Z_{T}$ by one. Thus, the relative boundary of $Z_T$ lies in the $(d-1)$-skeleton of $Z_\alpha$, and we can adjourn $Z_T$ as a new $d$-cell. Lemma \ref{LEM:Z-cell-properties}\eqref{ITEM:Z-disjoint-union} shows that this gives a cellular subdivision of $Z_\alpha$. 
\end{proof}

In addition, we obtain that $Z_\alpha$ is independent of some of the choices we made to define it.
\begin{cor} \label{COR: INDEPENDENT}
The cell complex $Z_\alpha$ is independent of the choice of the convex polygon $Q$ and the labels of $\alpha$ under cyclic rotation. More precisely:
\begin{enumerate}
\item\label{ITEM:QQ'} If $Q$ and $Q'$ are two convex $n_\alpha$-gons, then the cells of both $Z_{Q,\alpha}$ and $Z_{Q',\alpha}$ are labeled by the same set $\mathcal T_\alpha$, so that the map $(v_{t}(Q),w_{T}(Q))\mapsto (v_{t}(Q'),w_{T}(Q'))$ for $T\in\mathcal T_{\alpha}$ extended to convex hulls induces a cellular homeomorphism $Z_{Q,\alpha}\to Z_{Q',\alpha}$.
\item\label{ITEM:alpha^r} If $\alpha=(\alpha(1)\dots \alpha(n))$ is a list of labels (where $\alpha(j)\in \{\ii,\oo\}$ for $j=1,\dots, n$), and $\alpha^{\circlearrowright r}=(\alpha(r+1)\dots \alpha(n)\alpha(1)\dots \alpha(r))$ is the cyclic rotation by $0\leq r< n$ symbols, then there is a bijection $\tau_r:\mathcal T_{\alpha}\to \mathcal T_{\alpha^{\circlearrowright r}}$ given by cyclic rotation of the numbering of the external vertices. Then the map $(v_t,w_T)\mapsto (v_{\tau_r(t)},w_{\tau_r(T)})$ for $T\in\mathcal T_\alpha$ (with its induced map on a triangulation $t$) extended to convex hulls induces a cellular homeomorphism $Z_\alpha\to Z_{\alpha^{\circlearrowright r}}$.
\end{enumerate}
\end{cor}
\begin{proof}
For \eqref{ITEM:QQ'}, assume by induction, that we have defined maps $Z^{(d-1)}_{Q,\alpha}\to Z^{(d-1)}_{Q',\alpha}$ of the corresponding $(d-1)$-skeleta that restrict to a homeomorphism on each cell labeled by $T\in \mathcal T_\alpha$ of dimension less than $d$. Now, if $T$ labels a cell of dimension $d$, then the cells $Z_{T}(Q)$ and $Z_{T}(Q')$, for $Q$ and $Q'$ respectively, are closed balls of dimension $d$ by Corollary \ref{COR:ZT=KTxDT}, and, by Lemma \ref{LEM:Z-cell-properties}, $\rbd (Z_{T}(Q))=\coprod_{\tiny \begin{array}{cc} \text{$T'$ is edge ex-}\\ \text{pansion of $T$}\end{array}} \ri(K_{T'}(Q))\times \ri(\Delta_{T'}(Q))$ and $\rbd (Z_{Q',T})=\coprod_{\tiny \begin{array}{cc} \text{$T'$ is edge ex-}\\ \text{pansion of $T$}\end{array}} \ri(K_{T'}(Q'))\times \ri(\Delta_{T'}(Q'))$. By induction, we have homeomorphisms between the boundary $(d-1)$-spheres $\rbd(Z_{T}(Q))\to \rbd(Z_{T}(Q'))$, which may thus be extended to a homeomorphism of the $d$-balls $Z_{T}(Q)\to Z_{T}(Q')$.

The argument for \eqref{ITEM:alpha^r} is similar to the one for \eqref{ITEM:QQ'}, since the cyclic rotation $\tau_r:\mathcal T_{\alpha}\to \mathcal T_{\alpha^{\circlearrowright r}}$ given by the numbering of the external vertices respects edge expansions; in other words, the edge expansions of $\tau_r(T)$ are precisely $\tau_r(T')$, for edge expansions $T'$ of $T$. Thus, we can extend homeomorphisms of the $(d-1)$-skeleta $Z^{(d-1)}_\alpha\to Z^{(d-1)}_{\alpha^{\circlearrowright r}}$ to any $d$-cell labeled by $T$ as in \eqref{ITEM:QQ'}.
\end{proof}

Corollary~\ref{COR: INDEPENDENT} justifies referring to $Z_\alpha$ as a cell complex, which we will call the {\em assocoipahedron}, and which depends on an $\alpha$ up to cyclic rotation.

Figures~\ref{FIG:Pairahedra-cell} and~\ref{FIG:Higher-Poly-cell} show these cell structures for certain examples. The highest dimension cells and codimension one cells are labeled in each case. Compare to Figures~\ref{FIG:Pairahedra} and~\ref{FIG:Higher-Poly}.
\begin{figure}[h]
\[
 \includegraphics[scale=1]{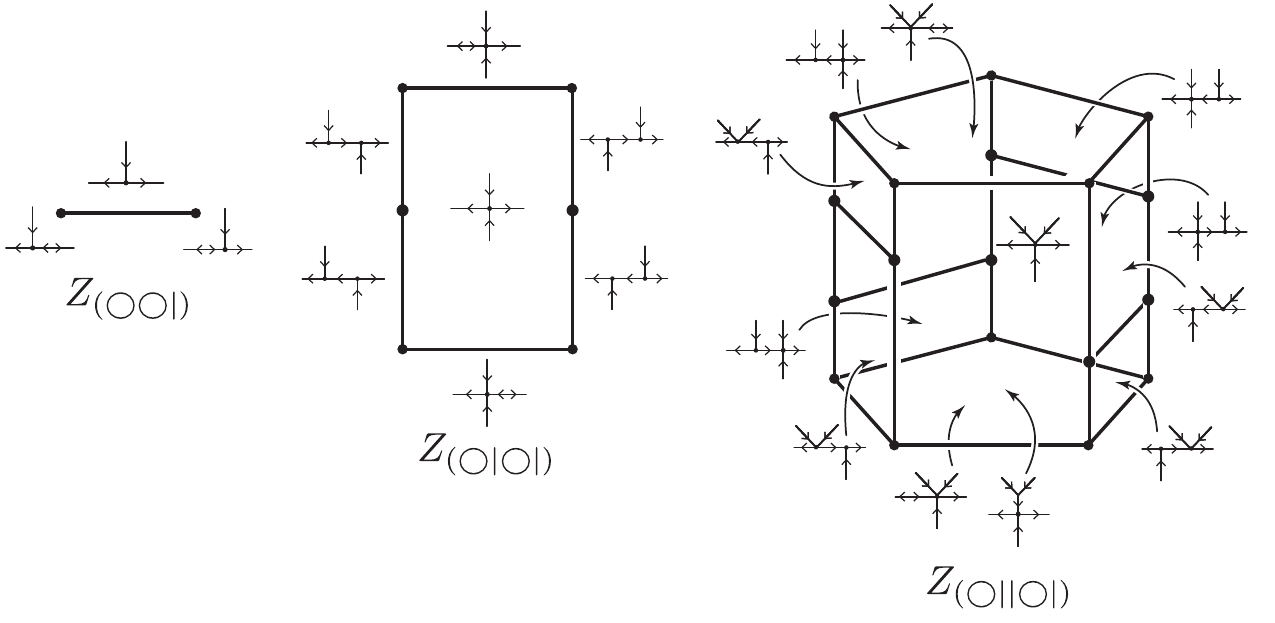}
\]
\caption{The cell complexes $Z_{(\oo\oo\ii)}$, $Z_{(\oo\ii\oo\ii)}$, and $Z_{(\oo\ii\ii\oo\ii)}$ (subdivision of $K_4\times \Delta^1$)}
\label{FIG:Pairahedra-cell}
\end{figure}
\begin{figure}[h]
\[
 \includegraphics[scale=1]{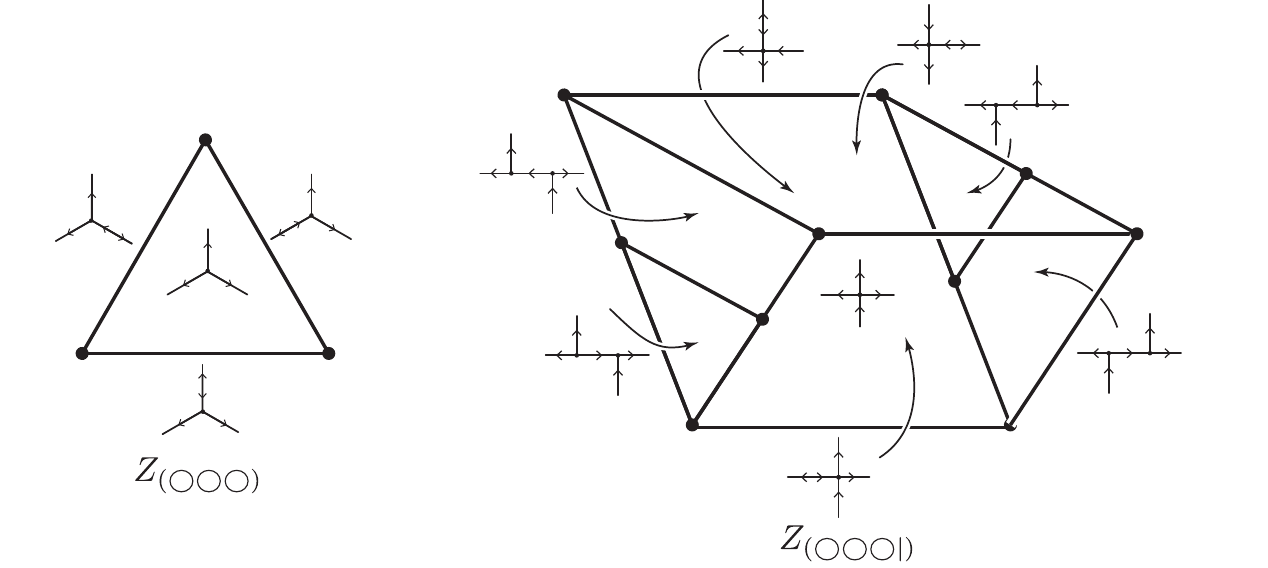}
\]
\caption{The cell complexes $Z_{(\oo\oo\oo)}\cong \Delta^2$ and $Z_{(\oo\oo\oo\ii)}$ (subdivision of $K_3\times \Delta^2$)}
\label{FIG:Higher-Poly-cell}
\end{figure}

We finish this section with a remark on an alternative approach for the construction of $Z_\alpha$.

\begin{rem}\label{REM:Loday-construction-Kalpha}
Although our construction of the assocoipahedron above used the secondary polytope to construct the associahedron, it would also work with other constructions of the associahedron. For example, Loday's construction of the associahedron, cf. \cite{L1}, is given for a maximally expanded Stasheff-type tree $S$ as follows. To each interior vertex $\ve$, associate $x_\ve=a_\ve\cdot b_\ve$ the product of the number of outgoing exterior edges to one side times the number of outgoing exterior edges to the other side. This gives a vector $v_S=\sum_\ve x_\ve e_\ve\in\R^{\#\text{ of interior vertices of $S$}}$ by ordering the internal vertices from left to right, see e.g. \cite[p.4]{L2}. An example is given in Figure \ref{FIG:Loday-construction}. The vectors lie in a hyperplane $\sum_\ve x_\ve=const$. The convex hull of the vectors $v_S$ then gives another representation of the associahedron $K_{n_\alpha-1}$; cf. \cite{L1}.

Now, an essential spine $E$ of a maximally expanded $(S,E)\in \mathcal{SE}_\alpha$ subdivides the Stasheff tree $S$ such that each subtree has exactly one outgoing exterior vertex $\ve_i$, where $i=1,\dots, k_\alpha$. For each such vertex $\ve_i$, we let $x_i$ be the sum of all numbers $x_\ve$ over all vertices $\ve$ that are still connected to $\ve_i$ after this the subdivision coming from $E$. Then, we can define $w_{(S,E)}\in \R^{k_\alpha}$ by setting (cf. Figure\ref{FIG:Loday-construction})
\[
w_{(S,E)}:=\sum_{i=1}^{k_\alpha} x_i e_i.
\]
This will yield an alternative geometric representation
\[
\Delta_\alpha=conv(\{w_{(S,E)}:(S,E)\text{ is maximally expanded}\}),
\]
and with this a space
$
\widetilde{Z}_\alpha:=K_{n_\alpha-1}\times \Delta_\alpha .
$
It is an easy exercise to see that $\Delta_\alpha$ is a $k_\alpha$-simplex and hence 
 $Z_\alpha$ and $\widetilde{Z}_\alpha$ are homeomorphic.
 A full treatment of the cellular structure of $\widetilde{Z}_\alpha$ using this approach requires an analogous definition of the spaces $Z_T$ as above.

\begin{figure}[h]
\[
 \includegraphics[scale=.9]{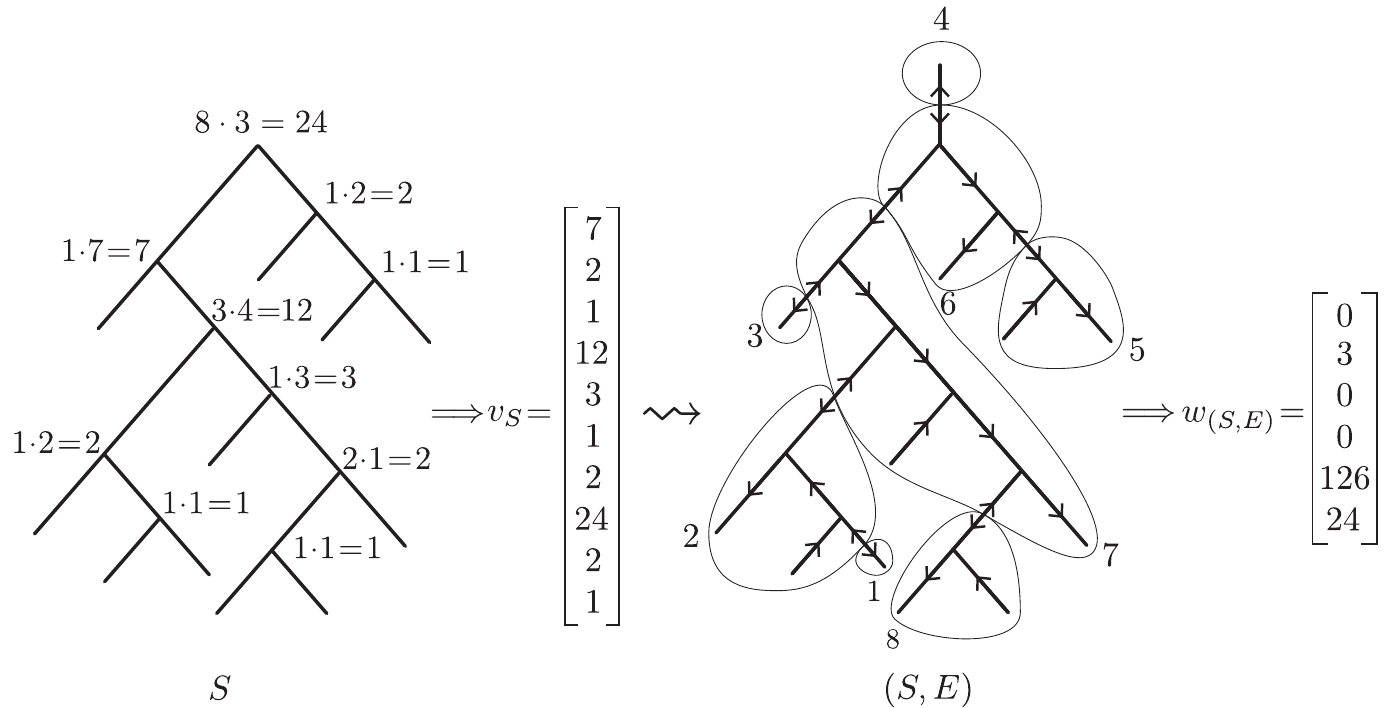}
\]
\caption{Realization of $Z_\alpha$ using Loday's construction}\label{FIG:Loday-construction}
\end{figure}
\end{rem}

\section{Vertices of the Directed Planar Tree Complex}\label{SEC:Vertices}

In this section, we perform some calculations of the number $C(\alpha)$ of vertices of the cell complex $Z_\alpha$. For $\alpha=(\oo\ii\ii\dots\ii\ii)$ with $n=n_\alpha$ labels, only one of which is outgoing $\oo$ and the rest being incoming $\ii$, this number is well known to be the Catalan number $C(\alpha)=C_{n-2}:=\frac 1 {n-1}{2(n-2)\choose n-2}$. For more general $\alpha$, we give a recursive relation for $C(\alpha)$ in Proposition \ref{PROP:C(alpha)-recursion}. We calculate these for the case of $\alpha=(\oo\ii\ii\dots\ii\ii\oo\ii\ii\dots\ii\ii)$ with exactly two outgoing labels $\oo$ in Proposition \ref{PROP:c-lm}.

\begin{defn}
Let $\alpha$ be a sequence of labels $\ii$ or $\oo$. Let $Z_\alpha$ be the cell complex defined in the last section. Then we define $C(\alpha)$ as the number of vertices of the cell complex $Z_\alpha$; i.e. $C(\alpha)$ is the number of maximally expanded $\alpha$-trees.
\end{defn}
The following proposition gives a recursive relation by which we can calculate $C(\alpha)$.

\begin{prop}\label{PROP:C(alpha)-recursion}
$C(\alpha)$ satisfies the following properties.
\begin{enumerate}
\item\label{C-cyclic}
Let $\alpha=(\alpha(1)\dots \alpha(n))$ be a list of labels (where $\alpha(j)\in \{\ii,\oo\}$ for all $j=1,\dots, n$), and denote by $\alpha^{\circlearrowright r}=(\alpha(r+1)\dots \alpha(n)\alpha(1)\dots \alpha(r))$ the cyclic rotation by $0\leq r< n$ symbols as in Corollary \ref{COR: INDEPENDENT}. Then, $C(\alpha^{\circlearrowright r})=C(\alpha)$.
\item\label{C-Catalan}
Let $\alpha=(\oo\ii\ii\dots\ii\ii)$ have a total of $n_\alpha\geq 2$ symbols, only one of which is outgoing. Then $C(\alpha)=C_{n_\alpha-2}$, where $C_n$ is the Catalan number $C_n:=\frac 1 {n+1}{2n\choose n}$.
\item\label{C-OO}
Let $\alpha=(\oo\oo)$. Then $C(\alpha)=1$.
\item\label{C-O+++}
Let $\alpha=(\alpha(1)\alpha(2)\dots \alpha(n))$ be a list of labels (where $\alpha(j)\in\{\ii,\oo\}$ for $j=1,\dots, n$). Assume that $n\geq 3$, that $\alpha(1)=\oo$, and that at least one of the $\alpha(j)$  for $j=2,\dots, n$ is also outgoing, $\alpha(j)=\oo$. Then,
\begin{eqnarray*}
C(\alpha)&=&C(\oo\alpha(2)\dots \alpha(n))\\
&=& C(\ii\alpha(2)\dots \alpha(n))+\sum_{j=2}^{n-1} C(\oo\alpha(2)\dots \alpha(j))\cdot C(\oo\alpha(j+1)\dots \alpha(n)).
\end{eqnarray*}
\end{enumerate}
\end{prop}
\begin{proof}
To see \eqref{C-cyclic}, note that the $0$-cells of $Z_\alpha$ and $Z_{\alpha^{\circlearrowright r}}$ are in bijective correspondence by Corollary \ref{COR: INDEPENDENT}, and thus, $C(\alpha)=C(\alpha^{\circlearrowright r})$.

 For \eqref{C-Catalan}, note that $Z_\alpha=K_{n_\alpha-1}$ (see Example \ref{EX:polyhedra}\eqref{EX:polyhedra-associahedra}), for which the number of vertices are well known to be $C_{n_\alpha-2}=\frac 1 {n_\alpha-1}{2(n_\alpha-2)\choose n_\alpha-2}$. Claim \eqref{C-OO} follows easily from the definition of $Z_{(\oo\oo)}=\{*\}$.

It remains to check claim \eqref{C-O+++}. Assume that $\alpha$ is as stated in  \eqref{C-O+++}, so that in particular, $\alpha(1)=\oo$, i.e. $\alpha=(\oo\alpha(2)\dots \alpha(n))$. Let $T$ be any maximally expanded $\alpha$-tree, and $f(T)=(S_T,E_T)$ as usual. Then all internal vertices of $S_T$ are trivalent. Note, there are two choices of how the ``root'' edge of $S_T$ (which, by definition, is the external edge connected to the outgoing edge of $S_T$) is labeled in $E_T$: either it is labeled by the symbol $\leftrightarrow$, or it is labeled as an outgoing edge.

In the first case, note that the trees $T$ that are possible with a $\leftrightarrow$ label at the root of $S_T$ are in one-to-one correspondence with the trees that have an incoming edge at the root. Thus, the number of maximally expanded such trees is precisely $C(\ii\alpha(2)\dots \alpha(n))$; see Figure \ref{FIG:root=leftright}.
\begin{figure}[h]
\[
 \includegraphics[scale=.9]{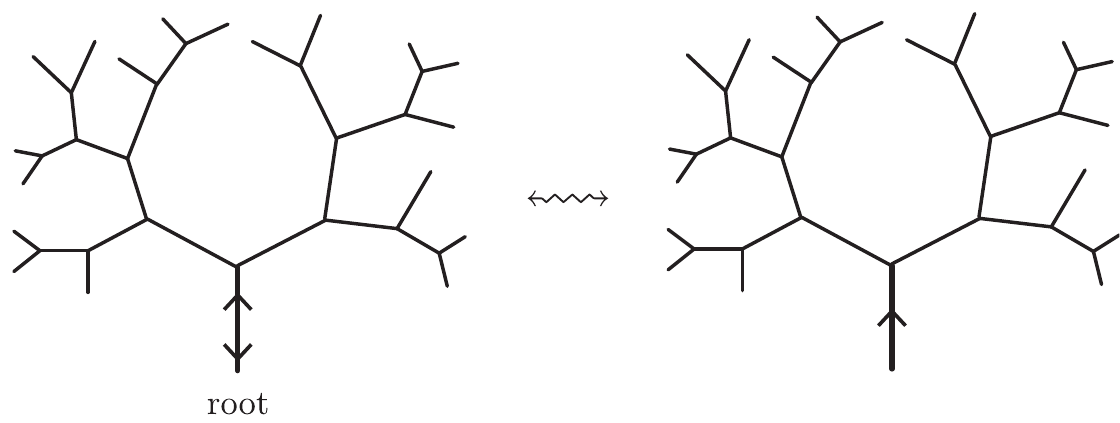}
\]
\caption{Root labeled with $\leftrightarrow$ corresponds to new labeling $(\ii\alpha(2)\dots \alpha(n))$}\label{FIG:root=leftright}
\end{figure}

Now, in the second case, the root of $S_T$ is labeled with the outgoing direction in $E_T$. 
Since $T$ is maximally expanded, there are exactly two edges that share a vertex with the root edge of $S_T$ and these edges must be directed toward this vertex in $E_T$.
Deleting the root edge from $T$ yields two new trees $T_1$ and $T_2$; see Figure \ref{FIG:root=outgoing}.
If there are $j$ external vertices in $T_1$ and $n-j+1$ external vertices in $T_2$, then those subtrees correspond exactly to the subtrees with labels $(\oo\alpha(2)\dots\alpha(j))$ and $(\oo\alpha(j+1)\dots\alpha(n))$ of which there are precisely $C(\oo\alpha(2)\dots\alpha(j))\cdot C(\oo\alpha(j+1)\dots\alpha(n))$ many.
\begin{figure}[h]
\[
 \includegraphics[scale=.9]{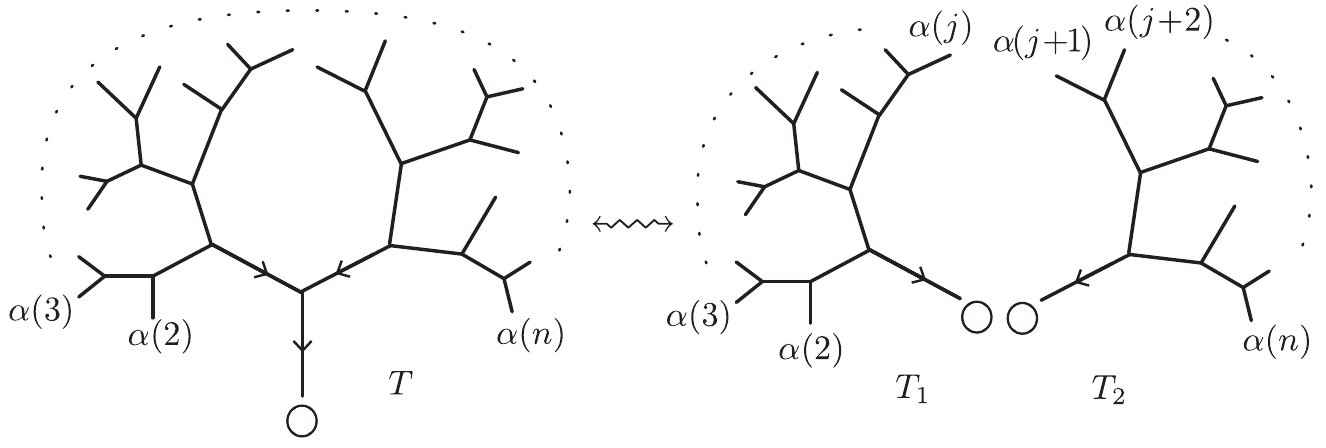}
\]
\caption{A root labeled with outgoing labels broken into a $(\oo\alpha(2)\dots\alpha(j))$-tree and a $(\oo\alpha(j+1)\dots\alpha(n))$-tree}\label{FIG:root=outgoing}
\end{figure}

Adding these two choices together gives precisely the claimed number in \eqref{C-O+++}.
\end{proof}

Since $C(\alpha)$ is cyclically invariant in $\alpha$ (Proposition \ref{PROP:C(alpha)-recursion}\eqref{C-cyclic}), it is determined by a sequence of $k_\alpha$ numbers $\ell_1,\dots,\ell_{k_\alpha}$, where $\ell_i\geq 0$ is the number of incoming labels between the $i$th and the $(i+1)$th outgoing labels; i.e.
\[
\alpha=(\oo\underbrace{\ii\ii\dots\ii\ii}_{\ell_1}\oo\underbrace{\ii\ii\dots\ii\ii}_{\ell_2}\dots \oo\underbrace{\ii\ii\dots\ii\ii}_{\ell_{k_\alpha}}).
\]
For these $\ell_i$, we call $c_{\ell_1,\dots,\ell_{k_\alpha}}:=C(\alpha)$ the generalized Catalan numbers. We clearly have that $c_{\ell_1,\dots,\ell_{k_\alpha}}=c_{\ell_{k_\alpha},\ell_1,\dots,\ell_{k_\alpha -1}}$. The precise relation with the Catalan numbers $C_n$ is given by
\[
c_\ell = C\Big((\oo\underbrace{\ii\ii\dots\ii\ii}_{\ell})\Big)=C_{\ell-1}=\frac{1}{\ell}{2(\ell-1) \choose \ell-1}.
\]

For $k_\alpha=2$, we will show below that $c_{\ell,m}$ is given by the following formula.
\begin{prop}\label{PROP:c-lm} We have:
\begin{eqnarray*}
c_{\ell,m}&=&c_{\ell+2}\cdot c_{m+2}\cdot \frac{(\ell+1)(\ell+2)(m+1)(m+2)}{2(\ell+m+1)(\ell+m+2)}\\
&=& {2(\ell+1)\choose \ell+1} {2(m+1)\choose m+1}\cdot \frac{(\ell+1)(m+1)}{2(\ell+m+1)(\ell+m+2)}
\end{eqnarray*}
\end{prop}
Before we can prove Proposition \ref{PROP:c-lm}, we first need to prove the basic Lemma \ref{LEM:Catalan-induction}. Let $b_{\ell,m}$ denote the numbers from the proposition, i.e. let 
\[
b_{\ell,m}:={2(\ell+1)\choose \ell+1} {2(m+1)\choose m+1}\cdot \frac{(\ell+1)(m+1)}{2(\ell+m+1)(\ell+m+2)}, \text{ for } \ell,m\geq 0.
\]
The claim of Proposition \ref{PROP:c-lm} is that $c_{\ell,m}=b_{\ell,m}$. It is immediate to check that in low cases we have:
\begin{equation}\label{EQU:b-low-cases}
 b_{\ell,0}=C_{\ell+1}, \quad \text{ and }\quad b_{\ell,1}={2(\ell+1) \choose \ell+1}\frac{6(\ell+1)}{(\ell+2)(\ell+3)}=2(C_{\ell+2}-C_{\ell+1}).
 \end{equation}
The numbers $b_{\ell,m}$ are closely related to the Catalan numbers, as the following lemma shows.
\begin{lem}\label{LEM:Catalan-induction}
For $p=0,\dots, N-1$, we have:
\begin{equation}\label{EQU:partial-catalan}
\sum_{j=0}^p C_j C_{N-j} = \frac 1 2\big(C_{N+1}+b_{N-p,p}-b_{N-p-1,p+1}\big).
\end{equation}
\end{lem}
\begin{proof}
We start with $p=0$. Since $C_0=1$ and, by \eqref{EQU:b-low-cases}, $b_{N,0}=C_{N+1}$ and $b_{N-1,1}=2(C_{N+1}-C_N)$, this shows that $C_0C_N=\frac 1 2 (C_{N+1}+b_{N,0}-b_{N-1,1})$ is correct.

Next, if \eqref{EQU:partial-catalan} is true for $p-1$, then, for $p$, we get
\[
\sum_{j=0}^p C_jC_{N-j}=\sum_{j=0}^{p-1} C_jC_{N-j}+C_pC_{N-p}=\frac 1 2\big(C_{N+1}+b_{N-p+1,p-1}-b_{N-p,p}\big)+C_pC_{N-p}.
\]
Thus the claim follows if we can show that $b_{N-p+1,p-1}-b_{N-p,p}+2C_pC_{N-p}=b_{N-p,p}-b_{N-p-1,p+1}$, or $b_{N-p+1,p-1}+b_{N-p-1,p+1}+2C_pC_{N-p}=2b_{N-p,p}$. A straightforward (but a bit lengthy) calculation of the left-hand side shows that
\begin{eqnarray*}
&&b_{N-p+1,p-1}+b_{N-p-1,p+1}+2C_pC_{N-p}\\
&=&{2(N-p+2)\choose N-p+2}{2p\choose p}\frac{(N-p+2)p}{2(N+1)(N+2)}\\
&&\hspace{1in} +{2(N-p)\choose N-p}{2(p+2)\choose p+2}\frac{(N-p)(p+2)}{2(N+1)(N+2)}\\
&&\hspace{1in} +2\cdot {2p\choose p}\frac{1}{p+1}\cdot {2(N-p)\choose N-p}\frac{1}{N-p+1}\\
&=&\frac{{2(N-p+1)\choose N-p+1}{2(p+1)\choose p+1}}{(N+1)(N+2)}\cdot \Bigg(\frac{(2N-2p+3)(p+1)p}{2(2p+1)}\\
&&\hspace{1in} +\frac{(N-p+1)(2p+3)(N-p)}{2(2N-2p+1)}+\frac{(N+1)(N+2)}{2(2p+1)(2N-2p+1)}\Bigg)\\
&=&\frac{{2(N-p+1)\choose N-p+1}{2(p+1)\choose p+1}}{(N+1)(N+2)}\cdot (N-p+1)(p+1),
\end{eqnarray*}
which is indeed $2\cdot b_{N-p,p}$.
\end{proof}
Note, that the previous lemma gives an inductive proof of the usual recursive relation for the Catalan numbers:
\begin{cor}
\[
\sum_{j=0}^N C_j C_{N-j} = C_{N+1}.
\]
\end{cor}
\begin{proof}
Set $p=N-1$ in \eqref{EQU:partial-catalan}. As before, using \eqref{EQU:b-low-cases}, we have that $b_{0,N}=C_{N+1}$ and $b_{1,N-1}=2(C_{N+1}-C_N)$. Thus, from \eqref{EQU:partial-catalan} in Lemma \ref{LEM:Catalan-induction}, we get:
\begin{equation*}
\sum_{j=0}^N C_jC_{N-j}=\sum_{j=0}^{N-1} C_jC_{N-j}+C_N=\frac 1 2 (C_{N+1}+b_{1,N-1}-b_{0,N})+C_N=C_{N+1}.
\end{equation*}
This concludes the proof.
\end{proof}
We are now ready to prove Proposition \ref{PROP:c-lm}.
\begin{proof}[Proof of Proposition \ref{PROP:c-lm}]
For the proof, we will not use the recursive relation for $c_{\ell,m}$ from Proposition \ref{PROP:C(alpha)-recursion}, but we will first give a different recursive relation, which works only for $\alpha$ with two outgoing labels $k_\alpha=2$, but has the advantage that we will be able to solve it explicitly.

Let $T$ be a maximally expanded $\alpha$-tree, with $\alpha=(\oo\underbrace{\ii\ii\dots\ii\ii}_{\ell}\oo\underbrace{\ii\ii\dots\ii\ii}_{m})$. Then, $E_T$ has exactly one label $\leftrightarrow$, which must separate the two outgoing edges as in Figure \ref{FIG:alpha(l,m)}, where we have placed the starting outgoing label to the far left.
\begin{figure}[h]
\[
 \includegraphics[scale=1]{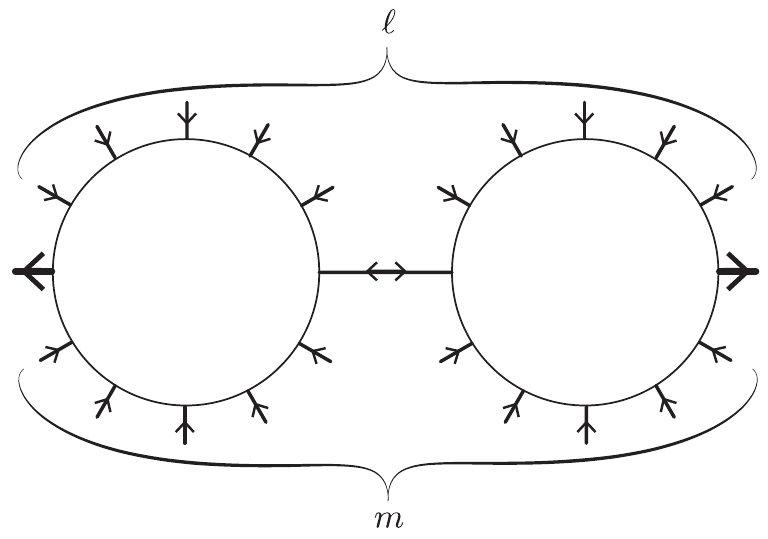}
\]
\caption{$\alpha$-tree $T$ with exactly one label $\leftrightarrow$ on $E_T$. The circles represent trees with only trivalent internal vertices and with directions given by a unique outward flow. There are $\ell$ incoming edges from the top, and $m$ incoming edges from the bottom.}\label{FIG:alpha(l,m)}
\end{figure}

We can convert this tree to a Stasheff-type tree with two more inputs as follows. Connect an edge at the edge labeled with $\leftrightarrow$, and make it the unique outgoing edge. We obtain a tree as depicted in Figure \ref{FIG:alpha(l,m)-->alpha(l+m+3)}.
\begin{figure}[h]
\[
 \includegraphics{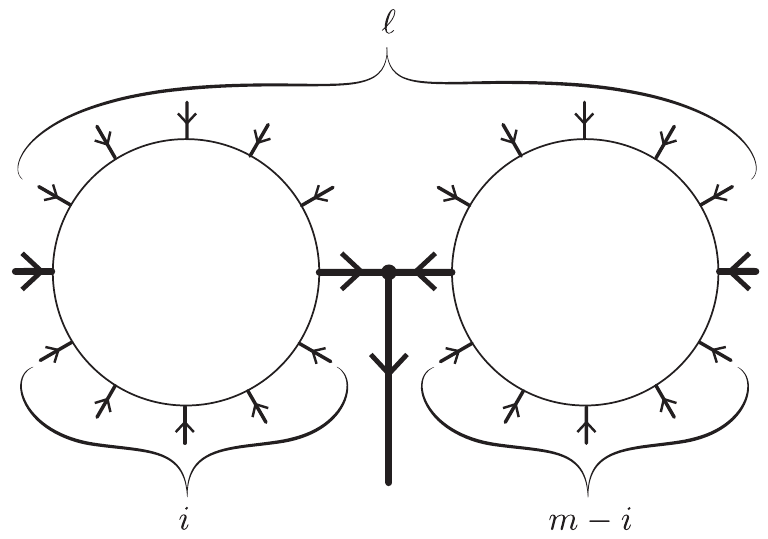}
\]
\caption{A new Stasheff-type tree obtained from the $\alpha$-tree by adding one more external edge, which will be marked as an outgoing edge}\label{FIG:alpha(l,m)-->alpha(l+m+3)}
\end{figure}

The new edge may appear at any position $i=1,\dots, m$ starting from the leftmost edge (which was the chosen first outgoing in $T$). Note, that in this manner, we obtain a tree with $\ell+m+2$ incoming edges, of which there are exactly $c_{\ell+m+2}$ many. However,  this over-counts the number of trees we are interested in. Some of the trees that we counted in $c_{\ell+m+2}$ do not appear as the modification of trees with exactly two outputs described above. These are the ones depicted in Figure \ref{FIG:alpha(l,m)-->not-in-alpha(l+m+3)}, where the corresponding edge labeled $\leftrightarrow$ would not divide the two outgoing edges of the original $\alpha$-tree $T$.
\begin{figure}[h]
\[
 \includegraphics[scale=1]{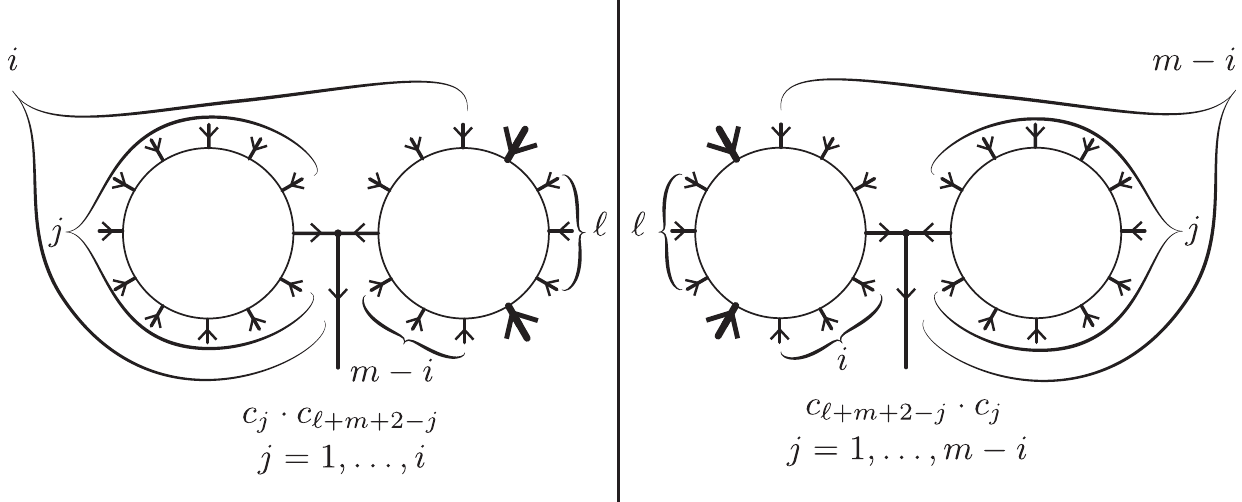}
\]
\caption{Stasheff-type trees that will not appear by adding a new outgoing edge. In these cases the original outgoing edges of $T$ are not separated by the $\leftrightarrow$-labeled edge. There are $j$ incoming edges on one side in either case, where either $j=1,\dots, i$ on the left, or $j=1,\dots,m-i$ on the right.}\label{FIG:alpha(l,m)-->not-in-alpha(l+m+3)}
\end{figure}

The possibilities depicted on the left in Figure \ref{FIG:alpha(l,m)-->not-in-alpha(l+m+3)} can be counted as $c_{j}\cdot c_{\ell+m+2-j}$, where $j=1,\dots, i$. The possibilities depicted on the right in Figure \ref{FIG:alpha(l,m)-->not-in-alpha(l+m+3)} can be counted as $c_{\ell+m+2-j}\cdot c_{j}$, where $j=1,\dots m-i$. We thus obtain a total number of maximally expanded $\alpha$-trees to be:
\begin{eqnarray*}
c_{\ell,m}&=&\sum_{i=0}^m\Big(c_{\ell+m+2}-\sum_{j=1}^i c_{j}\cdot c_{\ell+m+2-j}-\sum_{j=1}^{m-i} c_{\ell+m+2-j}\cdot c_{j}\Big)\\
&=&(m+1)c_{\ell+m+2}-\sum_{i=0}^m\Big( \sum_{j=1}^i c_{j}\cdot c_{\ell+m+2-j}+ \sum_{j=1}^{m-i} c_{\ell+m+2-j}\cdot c_{j}\Big).
\end{eqnarray*}
In particular, for $m=0$, we get $c_{\ell,0}=c_{\ell+2}=b_{\ell,0}$ for all $\ell\geq 0$. The claim of the proposition, namely that $c_{\ell,m}=b_{\ell,m}$ for all $m\geq 0$, thus follows from the following claim \eqref{EQ:induction-proof-step}:
\begin{equation}\label{EQ:induction-proof-step}
\forall m\geq 1, \ell\geq 0: \quad \sum_{i=0}^m\Big( \sum_{j=1}^i c_{j}\cdot c_{\ell+m+2-j}+ \sum_{j=1}^{m-i} c_{\ell+m+2-j}\cdot c_{j}\Big)= (m+1)c_{\ell+m+2}-b_{\ell,m}
\end{equation}
We prove \eqref{EQ:induction-proof-step} by induction on $m$. For $m=1$ and any $\ell\geq 0$, the left-hand side of \eqref{EQ:induction-proof-step} becomes $(0+c_{\ell+1+2-1}c_1)+(c_1c_{\ell+1+2-1}+0)=2c_{\ell+2}$. Noting that $b_{\ell,1}=2(C_{\ell+2}-C_{\ell+1})=2(c_{\ell+3}-c_{\ell+2})$, we see that the right-hand side is also $2c_{\ell+2}$.

Assume now that \eqref{EQ:induction-proof-step} holds for $m-1$ and any $\ell\geq 0$. Then the left hand side of \eqref{EQ:induction-proof-step} can be evaluated as
\begin{eqnarray*}
&& \sum_{i=0}^{m-1}\Big( \sum_{j=1}^i c_{j}\cdot c_{\ell+m+2-j}+ \sum_{j=1}^{m-1-i} c_{\ell+m+2-j}\cdot c_{j}\Big)\\
&&\hspace{.7in}+\Big(\sum_{j=1}^m c_{j}\cdot c_{\ell+m+2-j}\Big)+\Big(\sum_{i=0}^{m-1}  c_{\ell+m+2-(m-i)}\cdot c_{m-i}\Big)\\
& =&(m\cdot c_{\ell+m+2}-b_{\ell+1,m-1})+2\Big(\sum_{j=1}^{m} c_j c_{\ell+m+2-j}\Big)\\
& =& m\cdot c_{\ell+m+2}-b_{\ell+1,m-1}+2\Big(\sum_{j=0}^{m-1} C_{j} C_{\ell+m-j}\Big)\\
&\stackrel {\eqref{EQU:partial-catalan}} =& m\cdot c_{\ell+m+2}-b_{\ell+1,m-1}+(C_{\ell+m+1}+b_{\ell+1,m-1}-b_{\ell,m})\\
&=& (m+1)\cdot c_{\ell+m+2}-b_{\ell,m}.
\end{eqnarray*}
This proves the inductive step and thus proves the claim \eqref{EQ:induction-proof-step} for all $m$ and $\ell$.
\end{proof}

\begin{rem}
It would be interesting to have a formula for $c_{\ell_1,\ell_2,\ell_3}$ similar to the one for $c_{\ell_1,\ell_2}$ in Proposition \ref{PROP:c-lm}. Preliminary computations in this direction (yielding e.g. $c_{\ell,1,0}=c_{\ell+2}\frac{(\ell+1)12(7\ell^2+38\ell+50)}{(\ell+3)(\ell+4)(\ell+5)}$) indicate, that such a formula is more intricate than the one in Proposition \ref{PROP:c-lm}.

In fact, more than this---a closed formula for \emph{all} $c_{\ell_1,\dots,\ell_k}$---would be of interest.
\end{rem}

\end{document}